\documentclass[11pt]{amsart}

\usepackage{amssymb}
\usepackage{amsmath}
\usepackage{amscd}

\usepackage{xy}
\xyoption{all}

\usepackage{verbatim}
\newcommand{\nc}{\newcommand}
\newcommand{\rnc}{\renewcommand}

\nc{\exto}[1]{\stackrel{#1}{\longrightarrow}}
\nc{\dlim}{{\mathop{\lim\limits_{\longrightarrow}}}}
\nc{\lan}{\big\langle}
\nc{\ran}{\big\rangle}

\nc{\kk}{{\mathsf{k}}}
\nc{\ix}{{\mathsf{i}}}
\nc{\jx}{{\mathsf{j}}}

\nc{\C}{{\mathbb{C}}}
\nc{\HH}{{\mathbb{H}}}
\nc{\LL}{{\mathbb{L}}}
\nc{\PP}{{\mathbb{P}}}
\nc{\QQ}{{\mathbb{Q}}}
\nc{\RR}{{\mathbb{R}}}
\nc{\TT}{{\mathbb{T}}}
\rnc{\SS}{{\mathbb{S}}}
\nc{\ZZ}{{\mathbb{Z}}}

\nc{\CA}{{\mathcal{A}}}
\nc{\CB}{{\mathcal{B}}}
\nc{\CC}{{\mathcal{C}}}
\nc{\D}{{\mathcal{D}}}
\nc{\CE}{{\mathcal{E}}}
\nc{\CF}{{\mathcal{F}}}
\nc{\CG}{{\mathcal{G}}}
\nc{\CH}{{\mathcal{H}}}
\nc{\CJ}{{\mathcal{J}}}
\nc{\CK}{{\mathcal{K}}}
\nc{\CL}{{\mathcal{L}}}
\nc{\CM}{{\mathcal{M}}}
\nc{\CMI}{{\mathcal{MI}}}
\nc{\CMF}{{\mathcal{MF}}}
\nc{\CMFI}{{\mathcal{MFI}}}
\nc{\CMFB}{{\mathcal{MFB}}}
\nc{\CN}{{\mathcal{N}}}
\nc{\CO}{{\mathcal{O}}}
\nc{\CQ}{{\mathcal{Q}}}
\nc{\CP}{{\mathcal{P}}}
\nc{\CR}{{\mathcal{R}}}
\nc{\CS}{{\mathcal{S}}}
\nc{\CT}{{\mathcal{T}}}
\nc{\CU}{{\mathcal{U}}}
\nc{\CV}{{\mathcal{V}}}
\nc{\CW}{{\mathcal{W}}}
\nc{\CX}{{\mathcal{X}}}
\nc{\CY}{{\mathcal{Y}}}
\nc{\CZ}{{\mathcal{Z}}}
\nc{\CMo}{{\mathcal{M}^\circ}}
\nc{\Co}{{{C}^\circ}}

\nc{\BY}{{\overline{Y}}}
\nc{\BYD}{{\overline{Y}{}^{|D|}}}
\nc{\OZ}{{\overline{Z}}}
\nc{\bg}{{\bar{g}}}

\nc{\bq}{{\mathbf{q}}}
\nc{\BD}{{\mathbf{D}}}
\nc{\BG}{{\mathbf{G}}}
\nc{\BM}{{\mathbf{M}}}
\nc{\BP}{{\mathbf{P}}}
\nc{\BZ}{{\mathbf{Z}}}
\nc{\BPr}{{\mathsf{P}}}
\nc{\BR}{{\mathbf{R}}}
\nc{\BRO}[1]{{{\mathbf{R}}^{\circ}_{#1}}}
\nc{\BRD}[1]{{{\mathbf{R}}^{|D|}_{#1}}}
\nc{\BRP}[1]{{{\mathbf{R}}^{1}_{#1}}}
\nc{\BRTP}[1]{{{\mathbf{\tilde{R}}}{}^{1}_{#1}}}
\nc{\BS}{{\mathbf{S}}}
\nc{\BMS}{{{\mathbf{M}}^{{s}}}}
\nc{\BMSS}{{{\mathbf{M}}^{{ss}}}}
\nc{\BMZ}{{\mathbf{M}^{\circ}}}
\nc{\BCL}{{\mathbf{L}}}

\nc{\PCC}{{{}^\perp\CC}}

\nc{\ch}{{\mathsf{ch}}}
\nc{\td}{{\mathsf{td}}}
\nc{\Cl}{{\mathsf{Cliff}}}
\nc{\Clev}{{\mathop{\mathsf{Cliff}}^{\circ}}}
\nc{\FA}{{\mathfrak{A}}}
\nc{\FB}{{\mathfrak{B}}}
\nc{\FF}{{\mathfrak{F}}}
\nc{\FI}{{\mathfrak{I}}}
\nc{\FZ}{{\mathfrak{Z}}}

\nc{\TFA}{{\tilde{\mathfrak{A}}}}
\nc{\TFB}{{\tilde{\mathfrak{B}}}}

\nc{\fa}{{\mathfrak{a}}}
\nc{\fg}{{\mathfrak{g}}}
\nc{\fp}{{\mathfrak{p}}}
\nc{\FD}{{\mathfrak{D}}}
\nc{\FE}{{\mathfrak{E}}}
\nc{\FL}{{\mathfrak{L}}}
\nc{\FM}{{\mathfrak{M}}}
\nc{\FR}{{\mathfrak{R}}}
\nc{\FS}{{\mathsf{S}}}

\nc{\sfc}{{\mathsf{c}}}
\nc{\sfch}{{\mathsf{ch}}}
\nc{\sfh}{{\mathsf{h}}}

\nc{\SD}{{\mathsf{D}}}
\nc{\SK}{{\mathsf{K}}}
\nc{\SO}{{\mathsf{O}}}
\nc{\SQ}{{\mathsf{Q}}}
\nc{\SPV}{{\mathsf{S}^+\mathsf{V}}}
\nc{\SMV}{{\mathsf{S}^-\mathsf{V}}}
\nc{\SPMV}{{\mathsf{S}^\pm\mathsf{V}}}
\nc{\SX}{{S_X}}
\nc{\SY}{{S_Y}}
\nc{\phipsi}{{q}}
\nc{\eps}{\varepsilon}

\nc{\pim}{{\pi_-}}
\nc{\pip}{{\pi_+}}

\nc{\BE}{{{\mathbf E}}}
\nc{\TD}{{\widetilde{\D}}}
\nc{\TFD}{{\widetilde{\FD}}}
\nc{\TE}{{\tilde{E}}}
\nc{\TQ}{{\tilde{Q}}}
\nc{\TCA}{{\tilde{\CA}}}
\nc{\TCF}{{\tilde{\CF}}}
\nc{\TCG}{{\tilde{\CG}}}
\nc{\TCH}{{\tilde{\CH}}}
\nc{\TCL}{{\tilde{\CL}}}
\nc{\TF}{{\tilde{F}}}
\nc{\TW}{{\tilde{W}}}
\nc{\TCB}{{\widetilde{\CB}}}
\nc{\TCC}{{\tilde{\CC}}}
\nc{\TCX}{{\tilde{\CX}}}
\nc{\TCY}{{\tilde{\CY}}}
\nc{\TPhi}{{\tilde{\Phi}}}
\nc{\OPhi}{{\bar{\Phi}}}
\nc{\txi}{{\tilde{\xi}}}
\nc{\tp}{{\tilde{p}}}
\nc{\tq}{{\tilde{q}}}
\nc{\tzeta}{{\tilde{\zeta}}}
\nc{\tpi}{{\tilde{\pi}}}

\nc{\HCB}{{\widehat{\CB}}}
\nc{\HCU}{{\widehat{\CU}}}
\nc{\HE}{{\widehat{\CE}}}
\nc{\HS}{{\widehat{S}}}
\nc{\HX}{{\hat{X}}}
\nc{\HY}{{\hat{Y}}}
\nc{\HZ}{{\hat{Z}}}
\nc{\hxi}{{\hat{\xi}}}

\nc{\UH}{{\mathcal{H}}}

\nc{\TM}{{\widetilde{M}}}
\nc{\TCM}{{\widetilde{\CM}}}
\nc{\TS}{{\widetilde{S}}}
\nc{\TU}{{\widetilde{U}}}
\nc{\TX}{{\widetilde{X}}}
\nc{\TY}{{\widetilde{Y}}}
\nc{\TZ}{{\widetilde{Z}}}
\nc{\TYO}{{{\widetilde{Y}}^\circ}}
\nc{\barf}{{\bar{f}}}
\nc{\te}{{\tilde{e}}{}}
\nc{\tf}{{\tilde{f}}}
\nc{\tg}{{\tilde{g}}}
\nc{\ti}{{\tilde{\imath}}}
\nc{\tj}{{\tilde{\jmath}}}
\nc{\ty}{{\tilde{y}}}
\nc{\tphi}{{\tilde{\phi}}}
\nc{\hf}{{\hat{f}}}

\nc{\urho}{{\underline{\rho}}}

\nc{\LRA}{\Leftrightarrow}
\nc{\RA}{\Rightarrow}
\nc{\lotimes}{\mathbin{\mathop{\otimes}\limits^{\mathbb{L}}}}
\nc{\CEnd}{\mathop{\mathcal{E}\mathit{nd}}\nolimits}
\nc{\CExt}{\mathop{\mathcal{E}\mathit{xt}}\nolimits}
\nc{\CHom}{\mathop{\mathcal{H}\mathit{om}}\nolimits}
\nc{\RH}{\mathop{{\mathsf{R}}\Gamma}\nolimits}
\nc{\RGamma}{\mathop{{\mathsf{R}}\Gamma}\nolimits}
\nc{\RHom}{\mathop{\mathsf{RHom}}\nolimits}
\nc{\RCHom}{\mathop{\mathsf{R}\mathcal{H}\mathit{om}}\nolimits}
\nc{\RG}{\mathop{\mathsf{R\Gamma}}\nolimits}
\nc{\Hom}{\mathop{\mathsf{Hom}}\nolimits}
\nc{\Ext}{\mathop{\mathsf{Ext}}\nolimits}
\nc{\End}{\mathop{\mathsf{End}}\nolimits}
\nc{\Tor}{\mathop{\mathsf{Tor}}\nolimits}
\nc{\Tordim}{\mathop{\mathsf{Tor}\text{\rm-}\mathsf{dim}}\nolimits}
\nc{\Hilb}{\mathop{\mathsf{Hilb}}\nolimits}
\nc{\Spec}{\mathop{\mathsf{Spec}}\nolimits}
\nc{\Proj}{\mathop{\mathsf{Proj}}\nolimits}
\nc{\Pic}{\mathop{\mathsf{Pic}}\nolimits}
\renewcommand{\Im}{\mathop{\mathsf{Im}}\nolimits}
\nc{\Tw}{\mathop{\mathsf{Tw}}\nolimits}
\nc{\Cone}{\mathop{\mathsf{Cone}}\nolimits}
\nc{\Ker}{\mathop{\mathsf{Ker}}\nolimits}
\nc{\Coker}{\mathop{\mathsf{Coker}}\nolimits}
\nc{\codim}{\mathop{\mathsf{codim}}\nolimits}
\nc{\sing}{{\mathsf{sing}}}
\nc{\supp}{\mathop{\mathsf{supp}}}
\nc{\perf}{{\mathsf{perf}}}
\nc{\rank}{\mathop{\mathsf{rank}}}
\nc{\Pf}{{\mathsf{Pf}}}
\nc{\Gr}{{\mathsf{Gr}}}
\nc{\OGr}{{\mathsf{OGr}}}
\nc{\SGr}{{\mathsf{SGr}}}
\nc{\Flag}{{\mathsf{Fl}}}
\nc{\Kosz}{{\mathsf{Kosz}}}
\nc{\LGr}{{\mathsf{LGr}}}
\nc{\GTGr}{{\mathsf{G_2Gr}}}
\nc{\GT}{{\mathsf{G_2}}}
\nc{\GTF}{{\mathsf{G_2F}}}
\nc{\OF}{{\mathsf{OF}}}
\nc{\Fl}{{\mathsf{Fl}}}
\nc{\Bl}{{\mathsf{Bl}}}
\nc{\GL}{{\mathsf{GL}}}
\nc{\PGL}{{\mathsf{PGL}}}
\nc{\SL}{{\mathsf{SL}}}
\nc{\SP}{{\mathsf{Sp}}}
\nc{\Spin}{{\mathsf{Spin}}}
\nc{\Tot}{{\mathsf{Tot}}}
\nc{\ev}{{\mathsf{ev}}}
\nc{\od}{{\mathsf{odd}}}
\nc{\norm}{{\mathop{\mathrm{norm}}}}
\nc{\coev}{{\mathsf{coev}}}
\nc{\id}{{\mathsf{id}}}
\nc{\opp}{{\mathsf{opp}}}
\nc{\PS}{{{\PP^3}}}
\nc{\Qu}{{\mathsf{Q}}}
\nc{\tdim}{\mathop{\Tor\dim}}
\nc{\ecart}{{\fbox{$\scriptstyle\mathsf{EC}$}}}
\nc{\ad}{{\mathop{\mathsf ad}}}
\nc{\gr}{{\mathop{\mathsf gr}}}
\nc{\qgr}{{\mathop{\mathsf qgr}}}
\nc{\tor}{{\mathop{\mathsf tor}}}
\rnc{\mod}{{\mathop{\mathsf mod}}}
\nc{\Mod}{{\mathop{\mathsf Mod}}}
\nc{\Coh}{{\mathop{\mathsf Coh}}}
\nc{\Ab}{{\mathop{\mathcal{A}\mathit{b}}}}
\nc{\QCoh}{{\mathop{\mathsf QCoh}}}

\nc{\AAV}{{\mathcal{AAV}}}

\nc{\Rep}{{\mathsf{Rep}}}

\nc{\Cubics}{{{\mathcal{S}}_3}}
\nc{\VFT}{{{\mathcal{S}}_{14}}}
\nc{\VFTE}{{{\mathcal{N}}_{\mathrm{reg,sm}}}}
\nc{\MX}{{\CM_X}}
\nc{\MY}{{\CM_Y}}
\nc{\MYE}{{\CM_{Y,\CE}}}
\nc{\Yd}{{Y_d}}
\nc{\Yfive}{{Y_5}}
\nc{\Xg}{{X_{2g-2}}}
\nc{\Xtt}{{X_{22}}}
\nc{\Xst}{{X_{16}}}
\nc{\Xtw}{{X_{12}}}
\nc{\Xe}{{X_{8}}}
\nc{\Xf}{{X_{4}}}

\nc{\git}{{/\!\!/\!{}_\chi}}

\theoremstyle{plain}

\newtheorem{theorem}{Theorem}[section]
\newtheorem{conjecture}[theorem]{Conjecture}
\newtheorem{lemma}[theorem]{Lemma}
\newtheorem{proposition}[theorem]{Proposition}
\newtheorem{corollary}[theorem]{Corollary}

\theoremstyle{definition}

\newtheorem{definition}[theorem]{Definition}

\theoremstyle{remark}

\newtheorem{remark}[theorem]{Remark}

\title{Instanton bundles on Fano threefolds}
\author{Alexander Kuznetsov}
\address{\sloppy
\parbox{0.9\textwidth}{
Algebra Section, Steklov Mathematical Institute,
8 Gubkin str., Moscow 119991 Russia
\hfill\\[5pt]
The Poncelet Laboratory, Independent University of Moscow
\hfill\\[5pt]
Laboratory of Algebraic Geometry, SU-HSE\\[5pt]
}}
\email{akuznet@mi.ras.ru}
\date{}
\thanks{I was partially supported by
RFFI grants 10-01-93110, 10-01-93113, 11-01-00393, 11-01-00568, \hbox{11-01-92613-KO-a}, NSh-5139.2012.1,
the grant of the Simons foundation, and
by AG Laboratory SU-HSE, RF government grant, ag.11.G34.31.0023.}

\begin{document}

\begin{abstract}
We introduce the notion of an instanton bundle on a Fano threefold of index 2.
For such bundles we give an analogue of a monadic description and discuss
the curve of jumping lines. The cases of threefolds of degree 5 and 4 are
considered in a greater detail.
\end{abstract}

\maketitle

\section{Introduction}

The moduli space of stable bundles on the projective space $\PP^3$ is an important object of investigation
in algebraic geometry. Especially important subclass of stable bundles is constituted by the so-called {\em mathematical instanton bundles}.
By definition a mathematical instanton on $\PP^3$ is a stable vector bundle $E$ of rank 2 with $c_1(E) = 0$
and with the property that
$$
H^1(\PP^3,E(-2)) = 0,
$$
known as the {\em instantonic condition}. The second Chern class, $c_2(E)$ is known as the {\em charge},
or the {\em topological charge} of the instanton $E$.

Originally, instanton bundles appeared in the seminal work of Atiyah--Drinfeld--Hitchin--Manin \cite{ADHM}
as a way to describe Yang--Mills instantons on a four-sphere $S^4$ which play an important role in Yang--Mills
gauge theory. Since then they attracted a lot of attention, especially the questions like smoothness and connectedness
of their moduli space and different approaches to their construction were considered. Also a number of generalizations
of instantons appeared, such as instantons on higher-dimensional projective spaces \cite{OS,ST}
(in particular symplectic instantons) and noncommutative instantons \cite{KKO}.

The goal of this paper is to introduce another (in a way more direct) generalization of instantons.
Instead going to higher dimensions, or into the noncommutative world, we suggest just to replace
$\PP^3$ with another Fano threefold. In doing so we note that the line bundle $\CO_{\PP^3}(-2)$
appearing in the instantonic condition is nothing but the square root of the canonical bundle, so as soon
as we have a Fano threefold with canonical class being a square we can consider instantons on it.
This attracts our attention to Fano threefolds of index $2$.

Here we should also mention an independent paper of Daniele Faenzi~\cite{Fa},
which also discusses a generalization of instanton bundles to Fano threefolds,
especially to those with trivial third Betti number. In particular, the results
obtained in {\em loc. cit.} for the Fano threefold of index 2 and degree 5 and 4
are very close to the results in the present paper.

Recall that the index of a Fano manifold is the maximal integer dividing its canonical class.
By Fano--Iskovskikh--Mukai classification the index of a Fano threefold is bounded by $4$,
with $\PP^3$ being the only index $4$ variety, and the quadric $Q^3$ the only index $3$ variety.
Among the Fano threefolds of index $2$ the most important are those with Picard number 1.
Given such a threefold $Y$ we denote by $\CO_Y(1)$ the ample generator of the Picard group.
Then the canonical bundle of $Y$ is $\CO_Y(-2)$ and $\CO_Y(-1)$ is its square root.
So, we have the following

\begin{definition}[\cite{K03}]
Let $Y$ be a Fano threefold of index $2$.
An {\sf instanton bundle}\/ on $Y$ is a stable vector bundle $E$ of rank $2$ with $c_1(E) = 0$ such that
\begin{equation}\label{ins}
H^1(Y,E(-1)) = 0.
\end{equation}
The integer $c_2(E)$ is called the {\sf (topological) charge}\/ of the instanton $E$.
\end{definition}

The goal of this paper is to show that instantons on Fano threefolds of index 2 share many properties
of usual instantons. So, their investigation, interesting by itself, may be helpful for further
study of instantons on $\PP^3$. To be more precise we will concentrate on the following two issues:
the {\em monadic construction} and the {\em Grauert--M\"ulich Theorem}.

Recall that every instanton of charge $n$ on $\PP^3$ can be represented as the cohomology
in the middle term of a self-dual three-term complex
$$
\CO_{\PP^3}(-1)^{n} \to \CO_{\PP^3}^{2n+2} \to \CO_{\PP^3}(1)^{n}
$$
(known as a {\em monad}). The reason for this is a relatively simple structure of the bounded derived category $\D^b(\PP^3)$
of coherent sheaves on $\PP^3$. This category is known to have many full exceptional collections,
the most convenient for our question is the collection $(\CO_{\PP^3}(-1),\CO_{\PP^3},\CO_{\PP^3}(1),\CO_{\PP^3}(2))$.
The instantonic condition implies (by stability and Serre duality) that any instanton lies in the right orthogonal
to $\CO_{\PP^3}(2)$, which is the subcategory of $\D^b(\PP^3)$ generated by $\CO_{\PP^3}(-1)$, $\CO_{\PP^3}$, and $\CO_{\PP^3}(1)$.
Decomposing the instanton with respect to this collection gives the monad.

Of course, generic Fano threefold does not have a full exceptional collection, so the above description
cannot work verbatim. However, a certain part of it works. To be more precise, each Fano threefold $Y$
of index 2 has an exceptional collection $(\CO_Y,\CO_Y(1))$ (not full), which gives rise to a semiorthogonal
decomposition
$$
\D^b(Y) = \langle \CB_Y, \CO_Y, \CO_Y(1) \rangle,
$$
where triangulated category $\CB_Y$, defined as the orthogonal $\CB_Y = \langle \CO_Y,\CO_Y(1) \rangle^\perp$,
is called {\em the nontrivial component} of $\D^b(Y)$ and discussed in~\cite{K09}. Now if $E$
is an instanton of charge $n$ on $Y$ then analogously to the case of $\PP^3$ the instantonic condition implies
that $E$ is right orthogonal to $\CO_Y(1)$, hence it is contained in the subcategory $\langle \CB_Y,\CO_Y \rangle$
of $\D^b(Y)$. Decomposing $E$ with respect to this semiorthogonal decomposition we can see that the component 
with respect to $\CO_Y$ is just $\CO_Y^{n-2}$, while the component in~$\CB_Y$ is a very special vector bundle 
$\TE$ of rank $n$ which is called the {\sf acyclic extension} of the instanton $E$. The decomposition itself 
takes the form of a short exact sequence
$$
0 \to E \to \TE \to \CO_Y^{n-2} \to 0,
$$
which is an analogue of the monad. Moreover, the bundle $\TE$ itself should be considered as an analogue
of the {\em Kronecker module} (see e.g.~\cite{OSS}) associated to the instanton. We show that $\TE$ has many nice
properties, in particular it is self-dual with respect to a certain antiautoequivalence of the category $\CB_Y$,
which generalizes usual symmetry property of Kronecker modules. Moreover, we show that one can easily reconstruct
the instanton from its acyclic extension.

Another approach to construction and classification of instantons is based on investigation
of the behavior of the restriction of an instanton to lines. In the case of $\PP^3$ this behavior
is described by the classical Grauert--M\"ulich Theorem saying that if $E$ is an instanton of charge $n$ then
\begin{itemize}
\item for generic line $L \subset \PP^3$ one has $E_{|L} \cong \CO_L \oplus \CO_L$;
\item the lines $L \subset \PP^3$ for which the restriction $E_{L}$ is nontrivial ({\sf jumping lines})
are parameterized by a degree $n$ divisor $D_E$ in the Grassmannian $\Gr(2,4)$ of lines;
\item the divisor comes with a coherent sheaf (which is locally free of rank $1$
in points corresponding to lines $L$ such that $E_{|L} = \CO_L(1) \oplus \CO_L(-1)$),
and the instanton can be reconstructed from  the divisor and the associated sheaf.
\end{itemize}
We aim to prove the same for Fano threefolds of index 2. Of course, in this case we should look
at the Hilbert scheme of lines on $Y$ (which is traditionally called the {\em Fano scheme of lines}) $F(Y)$
which is a certain surface naturally associated to the threefold $Y$. It is not clear whether the analogue
of the first part of the Grauert--M\"ulich Theorem is true in this case, however the second definitely holds.
We show that as soon as the generic line on $Y$ is not a jumping line for an instanton $E$ of charge $n$,
the scheme of jumping lines is a curve $D_E$ on $F(Y)$ which is homologous to $nD_L$, where $D_L$ is the
curve on $F(Y)$ parameterizing lines intersecting a given line $L$. Moreover, we show that the curve
$D_E$ comes equipped with a coherent sheaf $\CL_E$ (locally free of rank 1 at the points corresponding
to 1-jumping lines) and discuss the question of reconstructing $E$ from the pair $(D_E,\CL_E)$.

The general study of instantons outlined above is illustrated by a more detailed description
of what goes on for Fano threefolds of index $2$ and degree $5$ and $4$ respectively.

In case of degree $5$ there is only one such threefold $Y_5$, it can be constructed
as a linear section of codimension 3 of the Grassmannian $\Gr(2,5)$ embedded into the Pl\"ucker space $\PP(\Lambda^2\kk^5)$.
% by a linear subspace of codimension 3. 
Such linear section is given by the corresponding three-dimensional
space of skew-forms in terms of which one can describe the geometry (and the derived category) of $Y_5$.
In particular, the nontrivial part $\CB_{Y_5}$ of the derived category of $Y_5$ is generated by
an exceptional pair of vector bundles (\cite{Or91}) which gives a description of the acyclic
extension $\TE$ of an instanton in terms of representations of the Kronecker quiver with 3 arrows
(which is a complete analogue of the Kronecker module describing instantons on $\PP^3$), and instanton
itself is described as the cohomology of a self-dual monad
$$
\CU^n \to \CO_{Y_5}^{4n+2} \to (\CU^*)^n,
$$
where $\CU$ is just the restriction of the tautological rank 2 vector bundle from the Grassmannian $\Gr(2,5)$.
On the other hand, the Fano scheme of lines on $Y_5$ is identified with $\PP^2$ and we show that the Kronecker
module above can be thought of as a net of quadrics parameterized by this $\PP^2$. In these terms the curve
$D_E$ of jumping lines of an instanton $E$ gets identified with the degeneration curve of the net of quadrics
and the associated sheaf $\CL_E$ with (the twist of) the corresponding theta-characteristic on $D_E$.
The usual procedure of reconstructing the net of quadrics from the associated theta-characteristic shows
that the instanton $E$ can be reconstructed from the pair $(D_E,\CL_E)$ in this case.

In case of degree 4 we also have a nice interpretation. Each Fano threefold $Y_4$ of index 2 and degree 4 is
an intersection of two quadrics in $\PP^5$. In the pencil of quadrics passing through $Y_4$ there are 6 degenerate
quadrics. We consider the double covering $C$ of $\PP^1$ (parameterizing quadrics in the pencil) ramified
in these 6 points. The curve $C$ has genus 2 and it is well known that $\CB_Y \cong \D^b(C)$ in this case
(see~\cite{BO1} or~\cite{K08a}). Let $\tau$ be the hyperelliptic involution of $C$. We show that the acyclic
extension $\TE$ of an instanton $E$ of charge $n$ on $Y_4$ corresponds under the above equivalence to
a semistable vector bundle $F$ on $C$ of rank $n$ such that $\tau^*F \cong F^*$ which has a special behavior
with respect to the Raynaud's bundle on $C$. Moreover, the Fano scheme of lines on $Y_4$ is isomorphic
(noncanonically) to the abelian surface $\Pic^0C$ and we show that the curve $D_E$ coincides with
the theta-divisor on $\Pic^0C$ associated with the bundle $F$. In particular, we show that in this case
one can reconstruct the instanton $E$ from the pair $(D_E,\CL_E)$.

The paper is organized as follows. In Section 2 we collect the preliminary material required for the rest
of the paper. In particular we discuss Fano threefolds of index 2 and their derived categories.
Section 3 is the central part of the paper where we develop the general theory of instantons.
In particular, we introduce the acyclic extension of an instanton and discuss the curve of its jumping lines.
In Section~\ref{sy5} we consider in detail the case of degree 5 Fano threefolds, and Section~\ref{sy4} deals
with degree 4 case. Finally, in Section~\ref{s_further} we outline possible approaches to instantons
on Fano threefolds of index 2 and degrees 3, 2, and 1.

\medskip

{\bf Acknowledgements.}
I am very grateful to Misha Verbitsky for suggesting me to give a talk at the conference ``Instantons in complex geometry'',
this was a great motivation to formulate my thoughts on the subject; and to Dima Markushevich for stimulating me
to write down the content of the talk from which this paper appeared. Also I would like to thank Mihnea Popa
for explaining me the state of art with theta-divisors of stable bundles and attracting my attention to Raynaud's bundles.
Finally, I am very grateful to the referees for many useful comments.

\section{Preliminaries}

We work over an algebraically closed field $\kk$ of characteristic $0$.

\subsection{Stable sheaves}

Let $F$ be a coherent sheaf on a smooth projective variety $X$ of dimension $n$. 
Assume a polarization (i.e. an ample divisor $H$ on $X$) is chosen.
Then {\sf the slope} of $F$ is defined as
\begin{equation*}
\mu_H(F) = c_1(F)\cdot H^{n-1}/r(F).
\end{equation*}
A sheaf $F$ is called {\sf Mumford-semistable}, or {\sf $\mu$-semistable} if for each subsheaf $G \subset F$
with $r(G) < r(F)$ one has $\mu_H(G) \le \mu_H(F)$. If the last inequality is strict for all such $G$ then
one says that $F$ is stable.

Analogously, $F$ is called {\sf Gieseker-semistable} if for each subsheaf $G \subset F$ with $r(G) < r(F)$
one has
\begin{equation*}
\chi(X,G(tH))/r(G) \le \chi(X,F(tH))/r(F)
\qquad\text{for $t \gg 0$}.
\end{equation*}
Here $\chi(X,-)$ stands for the Euler characteristic of a sheaf. By Riemann--Roch $\chi(X,F(tH))/r(F)$ is a polynomial
of degree $n$ with the coefficient at $t^n$ independent of $F$ and the coefficient at $t^{n-1}$ proportional to $\mu_H(F)$.
Thus each Mumford-stable sheaf is Gieseker-stable, and each Gieseker-semistable sheaf is Mumford-semistable.

Note also that rescaling of $H$ does not affect the (semi)stability of coherent sheaves. Thus if Neron--Severi group 
of $X$ is isomorphic to $\ZZ$ one can forget about the choice of polarization. Moreover, in this case one can
consider $c_1(F)$ just as an integer and the slope $\mu(F) = c_1(F)/r(F)$ as a rational number.
We are going to use this convention throughout the paper.

Note also that if the Picard group of $X$ is $\ZZ$ then a twisting of a sheaf $F$ by a line bundle
results in shifting the slope of $F$ by the integer equal to the class of this line bundle in $\Pic X$.
In particular, there is a unique twist such that the slope $\mu(F)$ is contained the interval $-1 < \mu(F) \le 0$.
This twist is called {\sf the normalized form} of $F$ and is denoted by $F_\norm$.

The following criterion is very useful for verification of stability.

\begin{lemma}[\cite{Hop}]\label{hoppe}
Assume that the Picard group of $X$ is $\ZZ$ and its ample generator $\CO_X(1)$ has global sections.
Let $F$ be a vector bundle of rank $r$ on $X$ such that for each $1\le k \le r-1$ the vector bundle 
$(\Lambda^kF)_\norm$ has no global sections. Then $F$ is stable.
\end{lemma}

We will refer to Lemma~\ref{hoppe} as {\sf Hoppe's criterion}.

\subsection{Fano threefolds of index 2}

A {\sf Fano variety} is a smooth projective variety $Y$ with the anticanonical class $-K_Y$ ample.
The {\sf index} of a Fano variety $Y$ is the maximal integer dividing the canonical class.
We refer to~\cite{IP} for a detailed introduction into the modern theory of Fano varieties.

It is well known that for a Fano variety of dimension $m$ the index does not exceed $m+1$ (see~\cite{Fu,IP}).
Moreover, there is only one Fano $m$-fold of index $m+1$, which is the projective space $\PP^m$,
and only one Fano $m$-fold of index $m$, which is the quadric $Q^m \subset \PP^{m+1}$.
In case of threefolds, thus we have $\PP^3$ of index 4 and $Q^3$ of index 3,
as well as Fano threefolds of index 2 and 1. All of them are classified in~\cite{IP}.
In this paper we restrict the attention to Fano threefolds of index 2 and the Picard group
of rank 1. There are five families of those, classified by the degree of the ample generator of the Picard group:
\begin{description}
\item[degree 5] $Y_5 = \Gr(2,5) \cap \PP^6 \subset \PP^9$ (a linear section of the Grassmannian);
\item[degree 4] $Y_4 = Q_1 \cap Q_2 \subset \PP^5$ (an intersection of two $4$-dimensional quadrics);
\item[degree 3] $Y_3 \subset \PP^4$ (a cubic threefold);
\item[degree 2] $Y_2 \to \PP^3$ (a quartic double solid);
\item[degree 1] $Y_1 \dashrightarrow \PP^2$ (a hypersurface of degree $6$ in the weighted projective space $\PP(1,1,1,2,3)$).
\end{description}
%\begin{itemize}
%\item degree 5, $Y_5 = \Gr(2,5) \cap \PP^6 \subset \PP^9$ (a linear section of the Grassmannian);
%\item degree 4, $Y_4 = Q_1 \cap Q_2 \subset \PP^5$ (an intersection of two $4$-dimensional quadrics);
%\item degree 3, $Y_3 \subset \PP^4$ (a cubic threefold);
%\item degree 2, $Y_2 \to \PP^3$ (a quartic double solid);
%\item degree 1, $Y_1 \dashrightarrow \PP^2$ (a hypersurface of degree $6$ in the weighted projective space $\PP(1,1,1,2,3)$).
%\end{itemize}

From now on we denote by $Y$ any Fano threefold of index 2. We will indicate the degree by a lower index, 
for example $Y_5$ will stand for the degree 5 threefold.
Since the Picard number of $Y$ is 1, it follows that
$$
H^2(Y,\ZZ) = H^4(Y,\ZZ) = H^6(Y,\ZZ) = \ZZ,
$$
(generated by the class of a hyperplane section, the class of a line, and the class of a point)
so the Chern classes of vector bundles can be thought of as integers. The ample generator
of the Picard group is denoted by $\CO_Y(1)$, so we have
$$
\omega_Y \cong \CO_Y(-2).
$$

\subsection{The Fano scheme of lines}\label{ss-fy}

The Hilbert scheme of lines on $Y$ is a surface which we denote by $F(Y)$ and it is called traditionally
the {\sf Fano scheme of lines on $Y$}. By definition, if $W^* = \Gamma(Y,\CO_Y(1))$ then $F(Y)$ is a subscheme
in $\Gr(2,W)$ consisting of all lines in $\PP(W)$ which lie in (the closure of) the image of $Y$
via the (rational) map given by the line bundle $\CO_Y(1)$.

For a line $L \subset Y$ we denote by $D_L \subset F(Y)$ the curve parameterizing lines intersecting $L$
and its class in the group $A^1(F(Y))$ of $1$-cycles on $F(Y)$ modulo rational equivalence (which we denote by $\sim$).

Let $Z$ denote the universal family of lines. It is a codimension 2 subscheme in $Y\times F(Y)$,
its fibers over $F(Y)$ are mapped onto lines in $Y$. Thus we have a diagram
$$
\xymatrix{
& Z \ar[dl]_q \ar[dr]^p \\ Y && F(Y)
}
$$
\begin{lemma}\label{qfib}
If a Fano threefold $Y$ of index $2$ is generic in its deformation class then 
the map $q$ in the above diagram is flat and finite.
\end{lemma}
\begin{proof}
In case of degree $d = 5$ and $d = 4$ it is easy to see that the map $q$ is finite and flat for any $Y_d$.
Indeed, if there is a point on $Y_d$ with infinite number of lines on $Y_d$ passing through this point
then these lines sweep in $Y_d$ a surface of degree less than $d$ which is impossible by the Lefschetz Theorem.
On the other hand, for $d \le 3$ one can verify the claim by a parameter counting.
\end{proof}

\begin{remark}
Although for generic $Y$ the map $q$ is flat and finite, both may fail for special 3-folds $Y$.
For example, consider the cubic 3-fold in $\PP^4 = \PP(x_0,,\dots,x_4)$ with equation
$x_0^2x_1 + x_1^3 + x_2^3 + x_3^3 + x_4^3 = 0$. It is easy to check that it is smooth.
However the lines passing throw the point $(1:0:0:0:0)$ are parameterized by the elliptic curve
$x_0 = x_1 = x_2^3 + x_3^3 + x_4^3 = 0$, so the fiber of $q$ over this point is not finite.
\end{remark}

On the other hand, the map $p:Z \to F(Y)$ is always flat and smooth. In fact, it is a projectivization
of the restriction to $F(Y)$ of the tautological bundle of $\Gr(2,W)$. We denote this rank 2 bundle
on $F(Y)$ by $M$. We will need to identify the first Chern class of $M$.

\begin{lemma}
We have $c_1(M) = -d D_L$.
\end{lemma}
\begin{proof}
For simplicity assume that $\CO_Y(1)$ is generated by global sections, i.e. the map $Y \dashrightarrow \PP(W)$
is regular. Take a subspace $W' \subset W$ of codimension $2$. Then $c_1(M^*)$ is represented by all lines $L \subset \PP(W)$
which intersect $\PP(W')$. In the other words it is the set of lines on $Y$ which pass through $Y \cap \PP(W')$.
But $Y \cap \PP(W')$ is a linear section of $Y$ of codimension 2, so its class is $c_1(\CO_Y(1))^2$ which is rationally
equivalent to $dL$, where $L$ is a line on $Y$. Hence the required set of lines is rationally equivalent
to $d$ times the set of lines intersecting $L$, that is to $dD_L$.
\end{proof}

\begin{corollary}\label{kzf}
We have $\omega_{Z/F(Y)} \cong p^*\CO_{F(Y)}(dD_L) \otimes q^*\CO_Y(-2)$ and $\omega_{Z/(Y\times F(Y))} \cong p^*\CO_{F(Y)}(dD_L)$.
\end{corollary}
\begin{proof}
Since $Z = \PP_{F(Y)}(M)$ we have $\omega_{Z/F(Y)} \cong p^*\det M^*\otimes\CO_{Z/F(Y)}(-2)$. The second formula follows immediately from
$\omega_{Z/(Y\times F(Y))} \cong \omega_{Z/F(Y)}\otimes q^*\omega_Y^{-1}$ since $\omega_Y \cong \CO_Y(-2)$ and $\CO_{Z/F(Y)}(1) = q^*\CO_Y(1)$.
\end{proof}

\subsection{Derived categories}

For an algebraic variety $X$ we denote by $\D^b(X)$ the bounded derived category of coherent sheaves on $X$.
It is a $\kk$-linear triangulated category. The shift functor in any triangulated category $\CT$ is denoted by $[1]$.
We denote $\Ext^p(F,G) = \Hom(F,G[p])$ and $\Ext^\bullet(F,G) = \oplus_{p \in \ZZ} \Ext^p(F,G)[-p]$.
One says that a triangulated category $\CT$ is {\sf $\Ext$-finite} if $\Ext^\bullet(F,G)$ is a finite dimensional graded vector space
for all $F,G \in \CT$. The derived category $\D^b(X)$ is $\Ext$-finite if $X$ is smooth and proper.

\begin{definition}[\cite{BK,BO1}]
A {\sf semiorthogonal decomposition}\/ of a triangulated category $\CT$ is a sequence of full triangulated
subcategories $\CA_1,\dots,\CA_m$ in $\CT$ such that $\Hom_{\CT}(\CA_i,\CA_j) = 0$ for $i > j$
and for every object $T \in \CT$ there exists a chain of morphisms
$0 = T_m \to T_{m-1} \to \dots \to T_1 \to T_0 = T$ such that
the cone of the morphism $T_k \to T_{k-1}$ is contained in $\CA_k$
for each $k=1,2,\dots,m$.
\end{definition}

A semiorthogonal decomposition with components $\CA_1,\dots,\CA_m$ is denoted $\CT = \langle \CA_1,\dots,\CA_m \rangle$.
The easiest way to produce a semiorthogonal decomposition is by using exceptional objects or collections.

\begin{definition}[\cite{B}]
An object $F \in \CT$ is called {\em exceptional}\/ if $\Ext^\bullet(F,F)=\kk$.
A collection of exceptional objects $(F_1,\dots,F_m)$ is called {\em exceptional}\/
if $\Ext^p(F_l,F_k)=0$ for all $l > k$ and all $p\in\ZZ$.
\end{definition}

The minimal triangulated subcategory of $\CT$ containing an exceptional object $F$ is equivalent
to the derived category of $\kk$-vector spaces. It is denoted by $\langle F \rangle$,
or sometimes just by $F$.

\begin{lemma}[\cite{BO1}]\label{sodgen}
If $\CT$ is an $\Ext$-finite triangulated category then any exceptional collection $F_1,\dots,F_m$ in $\CT$
induces a semiorthogonal decomposition
$$
\CT = \langle \CA , F_1, \dots, F_m \rangle
$$
where $\CA = \langle F_1, \dots, F_m \rangle^\perp = \{F \in \CT\ |\ \text{$\Ext^\bullet(F_k,F) = 0$ for all $1 \le k \le m$}\}$.
%and the other components are g
\end{lemma}

This construction can be efficiently applied to Fano varieties. Recall that by Kodaira vanishing
any line bundle on a Fano variety is exceptional. Moreover, if $X$ is a Fano variety of index $r$
then the sequence $\CO_X,\CO_X(1),\dots,\CO_X(r-1)$ is exceptional. In particular, for Fano threefolds
of index $2$ we have an exceptional pair $\CO_Y,\CO_Y(1)$. By Lemma~\ref{sodgen} it extends to a semiorthogonal decomposition
\begin{equation}\label{dby}
\D^b(Y) = \langle \CB_Y,\CO_Y,\CO_Y(1) \rangle,
\qquad
\CB_Y = \langle \CO_Y,\CO_Y(1) \rangle^\perp.
\end{equation}
The category $\CB_Y$ is called {\sf the nontrivial component} of $\D^b(Y)$.
Some of its properties are discussed in~\cite{K09}.

For each exceptional object $E \in \CT$ one can define the so-called mutation functors as follows.
For each object $F \in \CT$ consider the canonical evaluation map $\Ext^\bullet(E,F)\otimes E \to F$.
Its cone is denoted by $\LL_E(F)$ and is called {\sf the left mutation of $F$ through $E$}.
By definition we have a distinguished triangle
\begin{equation}\label{lmut}
\Ext^\bullet(E,F)\otimes E \to F \to \LL_E(F).
\end{equation}
The {\sf right mutation of $F$ through $E$} is defined dually, by using the coevaluation map and the distinguished triangle
\begin{equation}\label{rmut}
\RR_E(F) \to F \to \Ext^\bullet(F,E)^*\otimes E.
\end{equation}

The following fact is well known.

\begin{lemma}[\cite{BK}]\label{muteq}
The left and right mutations through $E$ vanish on the subcategory $\langle E\rangle$ and induce mutually inverse equivalences
$$
\xymatrix@1{{}^\perp E\  \ar@<2pt>[rr]^-{\LL_E} && \ E^\perp \ar@<2pt>[ll]^-{\RR_E} }
$$
\end{lemma}

\section{Instanton bundles}

Let $Y$ be a Fano threefold of index $2$.
Recall that by definition an {\sf instanton of charge $n$}\/ on $Y$ is
a stable vector bundle $E$ of rank $2$ with $c_1(E) = 0$, $c_2(E) = n$,
enjoying the instantonic condition~\eqref{ins}, which we rewrite for convenience
$$
H^1(Y,E(-1)) = 0.
$$

\subsection{Cohomology groups}

No wonder that the condition~\eqref{ins} has very similar consequences as the classical instanton condition on $\PP^3$.
For example, the cohomology table of $E$ has the same shape.

\begin{lemma}[\cite{K03}]\label{inst-coh}
Let $E$ be an instanton bundle of charge $n$ on a Fano threefold of index $2$ and degree~$d$. Then
the cohomology table of $E$ has the following shape
%$$
%\begin{array}{|c|c|c|c|c|c|}
%\hline
%t & \dots & -2 & -1 & 0 & \dots \\
%\hline
%\hline
%h^3(E(t)) & \dots & 0 & 0 & 0 & \dots \\
%\hline
%h^2(E(t)) & \dots & n-2 & 0 & 0 & \dots \\
%\hline
%h^1(E(t)) & \dots & 0 & 0 & n-2 & \dots \\
%\hline
%h^0(E(t)) & \dots & 0 & 0 & 0 & \dots \\
%\hline
%\end{array}
%$$
$$
\begin{array}{|c|c|c|c|c|c|c|c|}
\hline
t & \dots & -3 & -2 & -1 & 0 & 1 & \dots \\
\hline
\hline
h^3(E(t)) & \dots & * & 0 & 0 & 0 & 0 & \dots \\
\hline
h^2(E(t)) & \dots & * & n-2 & 0 & 0 & 0 & \dots \\
\hline
h^1(E(t)) & \dots & 0 & 0 & 0 & n-2 & * & \dots \\
\hline
h^0(E(t)) & \dots & 0 & 0 & 0 & 0 & * & \dots \\
\hline
\end{array}
$$
In particular,
$$
\begin{array}{ll}
H^0(E(t)) = 0 & \text{ for $t \le 0$,}\\
H^1(E(t)) = 0 & \text{ for $t \le -1$,}\\
H^2(E(t)) = 0 & \text{ for $t \ge -1$,}\\
H^3(E(t)) = 0 & \text{ for $t \ge -2$.}
\end{array}
$$
\end{lemma}
\begin{proof}
First note that $H^0(E(t)) = 0$ for $t \le 0$ by stability of $E$.
Further, by Serre duality
$$
H^3(E(t))^* = H^0(E^*(-t-2)) = H^0(E(-t-2)) = 0
$$
for $t \ge -2$.
Also by the Serre duality we have $H^2(E(-1))^* = H^1(E^*(-1)) = H^1(E(-1)) = 0$.
Finally, consider the Koszul complex
$$
0 \to \CO(-3) \to \CO(-2)^3 \to \CO(-1)^3 \to \CO \to \CO_Z \to 0,
$$
given by a triple of global sections of $\CO(1)$ with $Z$ a zero-dimensional subscheme of $Y$ of length $d$
(note that $\dim H^0(Y_d,\CO(1)) = d + 2 \ge 3$, so we can always find a triple of sections).
Note that $E\otimes\CO_Z$ is an artinian sheaf, in particular $H^{>0}(E\otimes\CO_Z) = 0$.
On the other hand, looking at the hypercohomology spectral sequence of the above Koszul complex
tensored with $E$ we see that $H^2(E)$ cannot be killed by anything (since $H^2(E(-1)) = H^3(E(-2)) = 0$),
hence if $H^2(E) \ne 0$ it should contribute nontrivially into $H^2(E\otimes\CO_Z) = 0$. Thus $H^2(E) = 0$.
Twisting additionally by $\CO(t)$ with $t \ge 0$ and using the same argument we prove inductively
that $H^2(E(t)) = 0$ for all $t \ge 0$. Then by Serre duality we have $H^1(E(-2-t)) = 0$.
This explains all zeros in the table. Applying Riemann--Roch one can easily deduce that
$\dim H^1(E) = \dim H^2(E(-2)) = n-2$.
\end{proof}

\begin{corollary}
The charge of an instanton bundle is greater or equal than $2$.
\end{corollary}

The instanton bundles of charge 2 are called {\sf the minimal instantons}.
They are particularly interesting. For example they have the following vanishing property.

\begin{corollary}
If $E$ is a minimal instanton then $H^i(E(t)) = 0$ for all $i$ and $-2 \le t \le 0$.
\end{corollary}

\begin{remark}\label{qu1}
The possible values of $\dim H^0(E(1)) = \dim H^3(E(-3))$ and $\dim H^1(E(1)) = \dim H^2(E(-3))$
are hard to find. There is a simple restriction
$$
\dim H^0(E(1)) - \dim H^1(E(1)) = 2d - 2n + 4
$$
which is given by Riemann--Roch. Moreover, probably one can show that
$$
\dim H^0(E(1)) \le 2d,
\qquad
\dim H^1(E(1)) \le 2n - 4.
$$
For this it is enough to check that for generic linear section $C$ of $Y$ of codimension $2$
(which is an elliptic curve) one has $H^0(C,E_{|C}) = 0$. In this case it would be easy to deduce
for minimal instantons the equalities $H^\bullet(E(1)) = \kk^{2d}$, $H^\bullet(E(-3)) = \kk^{2d}[-3]$.
%
%{\tt THE INEQUALITY $h^1(E(1)) \le 2(n-2)$ IS NOT CLEAR!}
\end{remark}

\subsection{The acyclic extension}

As we have seen in Lemma~\ref{inst-coh} each instanton $E$ enjoys the vanishing
$$
H^\bullet(Y,E(-1)) = 0.
$$
One can easily produce from $E$ another bundle which has a stronger vanishing.

%Recall that in~\cite{K09} we associated with each Fano threefold $Y$ of index 2 an admissible triangulated subcategory
%$$
%\CB_Y = \langle \CO_Y, \CO_Y(1) \rangle^\perp = \{ F \in \D^b(Y)\ |\ H^\bullet(Y,F) = H^\bullet(Y,F(-1)) = 0 \} \subset \D^b(Y).
%$$
%Each instanton can be considered as an object of the category $\CO(1)^\perp \subset \D^b(Y)$
%by Lemma~\ref{inst-coh}. It turns out to be useful to consider also its projection onto the subcategory $\CB_Y$.

\begin{lemma}
For each instanton bundle $E$ there exists a unique short exact sequence
\begin{equation}\label{te}
\xymatrix@1{0 \ar[r] & E \ar[r]^-{\lambda_E} & \TE \ar[r] & \CO_Y^{n-2} \ar[r] & 0 }
\end{equation}
such that $\TE$ is acyclic, i.e.
$$
H^\bullet(Y,\TE) = 0.
$$
\end{lemma}
\begin{proof}
Indeed, it is clear that $\TE$ is nothing but the universal extension of $H^1(Y,E)\otimes\CO_Y$ by $E$.
\end{proof}

Another way to describe $\TE$ is by saying that
\begin{equation*}
\TE = \LL_{\CO_Y}E,
\end{equation*}
the left mutation of $E$ through $\CO_Y$. Indeed, the definition of the left mutation~\eqref{lmut}
in this case literally coincides with exact sequence~\eqref{te}.

The bundle $\TE$ will be referred to as {\sf the acyclic extension}\/ of the instanton $E$.
Recall the semiorthogonal decomposition~\eqref{dby} of $\D^b(Y)$. We have the following

\begin{lemma}\label{te1}
The acyclic extension of an instanton of charge $E$ is a simple $\mu$-semistable vector bundle $\TE$ on $Y$ 
with
$$
r(\TE) = n,\quad
c_1(\TE) = 0,\quad
c_2(\TE) = n,\quad
c_3(\TE) = 0,\qquad
H^\bullet(\TE) = H^\bullet(\TE(-1)) = 0.
$$
In particular, $\TE \in \CB_Y$. Moreover,
$$
h^0(\TE^*) = h^1(\TE^*) = n-2,\qquad
h^2(\TE^*) = h^3(\TE^*) = 0.
$$
\end{lemma}
\begin{proof}
Chern classes and cohomology of $\TE$ are computed immediately using the defining sequence~\eqref{te}.
To compute the cohomology of $\TE(-1)$ we twist~\eqref{te} by $-1$, and to compute
the cohomology of $\TE^*$ we dualize~\eqref{te} and use self duality of $E$.

To check that $\TE$ is simple we first show that $\Hom(E,\TE) = \kk$ (by applying $\Hom(E,-)$ to~\eqref{te}
and noting that $E$ itself is simple and $\Hom(E,\CO_Y) = H^0(Y,E) = 0$). Then applying $\Hom(-,\TE)$ to~\eqref{te}
we see that $\TE$ is simple. Finally, to establish $\mu$-semistability of $\TE$ we note that $\TE$ is an extension
of two $\mu$-semistable sheaves of the same slope.
\end{proof}

\subsection{The antiautoequivalence}

Recall that any instanton, being a rank 2 bundle with trivial determinant, is self dual.
This self duality translates to the following property of the acyclic extension.

Consider the following antiautoequivalence of the category $\CO_Y^\perp \subset \D^b(Y)$.
First, note that the duality functor
$$
F \mapsto \RCHom(F,\CO_Y)
$$
gives an antiequivalence of the category $\CO_Y^\perp$ onto the category ${}^\perp\CO_Y$.
Composing it with the {\sf left mutation}\/ functor $\LL_\CO$ with respect to $\CO_Y$,
and using Lemma~\ref{muteq} we conclude that
$$
\SD:\CO_Y^\perp \to \CO_Y^\perp,\qquad
F \mapsto \LL_\CO(\RCHom(F,\CO_Y))
$$
is an antiautoequivalence of $\CO_Y^\perp$. Moreover, it is easy to see that $\SD$ is involutive.

\begin{lemma}
We have a functorial isomorphism $\delta:\SD^2 \stackrel\sim\to \id$.
\end{lemma}
\begin{proof}
Indeed, for each $F$ we have a canonical distinguished triangle
$$
\RHom(F,\CO_Y)\otimes\CO_Y \to \RCHom(F,\CO_Y) \to \SD(F)
$$
Dualizing it we obtain a triangle
$$
\RCHom(\SD(F),\CO_Y) \to F \to \RHom(F,\CO_Y)^*\otimes\CO_Y.
$$
Since $\LL_\CO(\CO_Y) = 0$, the application of exact functor $\LL_\CO$ gives a functorial isomorphism $\SD^2(F) \cong \LL_\CO(F)$.
But if $F \in \CO_Y^\perp$ then $\LL_\CO(F) = F$.
%\begin{multline*}
%\qquad
%\SD^2(F) \cong
%\LL_\CO(\RCHom(\LL_\CO(\RCHom(F,\CO_Y)),\CO_Y)) \cong
%\\ \cong
%\LL_\CO(\RCHom(\RCHom(\RR_\CO(F),\CO_Y),\CO_Y)) \cong
%\LL_\CO(\RR_\CO(F)) \cong
%F,
%\qquad
%\end{multline*}
%where $\RR_\CO$ stands for the right mutation functor.
\end{proof}

Moreover, the antiautoequivalence $\SD$ preserves the subcategory $\CB_Y$.

\begin{proposition}
The category $\CB_Y$ is preserved by the antiautoequivalence $\SD$.
\end{proposition}
\begin{proof}
Assume that $F \in \CB_Y = \langle \CO_Y,\CO_Y(1) \rangle^\perp$.
Then we have $\RCHom(F,\CO_Y) \in {}^\perp\langle \CO_Y(-1),\CO_Y \rangle$ and so
$\SD(F) = \LL_\CO(\RCHom(F,\CO_Y)) \in {}^\perp\CO_Y(-1) \cap \CO_Y^\perp$.
But since $\omega_Y \cong \CO_Y(-2)$ it follows from the Serre duality that ${}^\perp\CO_Y(-1) = \CO_Y(1)^\perp$,
so we see that $\SD(F) \in \CO_Y^\perp \cap \CO_Y(1)^\perp = \langle \CO_Y,\CO_Y(1) \rangle^\perp = \CB_Y$.
\end{proof}

%Note that by construction of the acyclic extension of an instanton, it is an object of $\CB_Y$.
%In some cases we have an explicit description of the category $\CB_Y$, it is instructive to see
%what kind of objects we obtain there. But before this we observe the following

\subsection{The self-duality of acyclic extensions}

Now we can state the self duality property of $\TE$.

\begin{proposition}\label{dte}
If $\TE$ is the acyclic extension of an instanton then there is a canonical isomorphism $\phi:\SD(\TE) \to \TE$.
Moreover, the isomorphism $\phi$ is skew-symmetric, that is the diagram
% $$
% \xymatrix{
% \SD(\TE) \ar[rr]^{\SD(\phi)} \ar@{=}[d] && \SD^2(\TE) \ar[d]^{\delta_\TE} \\
% \SD(\TE) \ar[rr]^{-\phi} && \TE
% }
% $$
% $$
% \xymatrix@R=7pt{
% && \SD^2(\TE) \ar[dd]^{\delta_\TE} \\
% \SD(\TE) \ar[rru]^-{\SD(\phi)} \ar[rrd]_-{-\phi} \\
% && \TE
% }
% $$
$$
\xymatrix{
& \SD(\TE) \ar[ld]_-{\SD(\phi)} \ar[rd]^-{-\phi} \\
\SD^2(\TE) \ar[rr]^-{\delta_\TE} && \TE
}
$$
commutes.
%map $\SD(\phi):\SD(\TE) \to \SD^2(\TE)$ coincides with $\phi$
%after canonical identification $\SD^2(\TE) = \TE$.
\end{proposition}
\begin{proof}
Applying $\RCHom(-,\CO_Y)$ to~\eqref{te} and denoting by $\sigma:E^* \to E$ the canonical isomorphism we obtain an exact sequence
\begin{equation}\label{ted}
\xymatrix@1{0 \ar[r] & \CO_Y^{n-2} \ar[r] & \RCHom(\TE,\CO_Y) \ar[r]^-{\sigma\lambda_E^T} & E \ar[r] &  0.}
\end{equation}
Combining it with~\eqref{te} we obtain a long exact sequence
$$
\xymatrix@1{0 \ar[r] & \CO_Y^{n-2} \ar[r] & \RCHom(\TE,\CO_Y) \ar[rr]^-{\lambda_E\sigma\lambda_E^T} && \TE \ar[r] &  \CO_Y^{n-2} \ar[r] & 0.}
$$
Since $\LL_\CO(\CO_Y) = 0$, we see that
$$
\phi := \LL_\CO(\lambda_E\sigma\lambda_E^T):\SD(\TE) \to \LL_\CO(\TE) = \TE
$$
is an isomorphism.
Let us show that $\phi$ is skew-symmetric. For this note that the above arguments give the following commutative diagram
$$
% \xymatrix{
% \RCHom(\TE,\CO_Y) \ar[rr]^-{\lambda_E\sigma\lambda_E^T} \ar[d] && \TE \ar@{=}[d] \\
% \SD(\TE) \ar[rr]^-{\phi} && \TE
% }
\xymatrix@!C{
\RCHom(\TE,\CO_Y) \ar[rd]_{\lambda_E\sigma\lambda_E^T} \ar[rr] && \SD(\TE) \ar[ld]^\phi \\
& \TE
}
$$
Dualizing it we obtain
$$
% \xymatrix{
% \TE && \RCHom(\TE,\CO_Y) \ar[ll]_-{\lambda_E\sigma^T\lambda_E^T} \ar@{=}[d] \\
% \RCHom(\SD(\TE),\CO_Y) \ar[u] && \RCHom(\TE,\CO_Y) \ar[ll]_-{\phi^T}
% }
\xymatrix@!C{
\TE && \RCHom(\SD(\TE),\CO_Y) \ar[ll] \\
& \RCHom(\TE,\CO_Y) \ar[lu]^{\lambda_E\sigma^T\lambda_E^T} \ar[ru]_{\phi^T}
}
$$
and applying $\LL_\CO$ we obtain
$$
% \xymatrix{
% \TE && \SD(\TE) \ar[ll]_-{\LL_\CO(\lambda_E\sigma^T\lambda_E^T)} \ar@{=}[d] \\
% \SD^2(\TE) \ar[u] && \SD(\TE) \ar[ll]_-{\SD(\phi)}
% }
\xymatrix@!C{
\TE && \SD^2(\TE) \ar[ll] \\
& \SD(\TE) \ar[lu]^{\LL_\CO(\lambda_E\sigma^T\lambda_E^T)} \ar[ru]_{\SD(\phi)}
}
$$
Now it remains to note that the arrow in the top row is $\delta_\TE$, and since $\sigma^T = - \sigma$, the left arrow is $-\phi$.
\end{proof}

\subsection{Reconstruction of the instanton}

It turns out that any vector bundle $F$ satisfying properties of both Lemma~\ref{te1} and Proposition~\ref{dte} is the acyclic extension
of appropriate instanton.

\begin{theorem}\label{ftoe}
Assume that $F$ is a vector bundle on $Y$ with
$$
\begin{array}{l}
r(F) = n,\quad
c_1(F) = 0,\quad
c_2(F) = n,\quad
c_3(F) = 0,\\
H^\bullet(F) = H^\bullet(F(-1)) = 0,\\
\SD(F) \cong F.
\end{array}
$$
Then $H^i(Y,F^*) = 0$ unless $i = 0,1$ and $h^0(F^*) = h^1(F^*) \le n-2$.

Moreover, if $h^0(F^*) = n-2$ then there is a unique instanton $E$ of charge $n$ such that $F \cong \TE$.
\end{theorem}

\begin{remark}\label{k0b}
It is easy to see that the conditions $H^\bullet(F) = H^\bullet(F(-1)) = 0$ together with $c_1(F) = 0$ imply
$c_2(F) = r(F)$ and $c_3(F) = 0$. Indeed, it follows easily from the description of the numerical Grothendieck
group of the category $\CB_Y$, see~\cite{K09}.
\end{remark}

\begin{proof}
Let us write down the condition $\SD(F) \cong F$ explicitly.
Since $F$ is a vector bundle, we have $\RCHom(F,\CO_Y) \cong F^*$.
Hence $\SD(F) = \Cone(H^\bullet(Y,F^*)\otimes\CO_Y \to F^*)$.
Writing down the long exact sequence of sheaf cohomology we obtain a long exact sequence
$$
0 \to H^0(Y,F^*)\otimes\CO_Y \to F^* \to F \to H^1(Y,F^*)\otimes\CO_Y \to 0
$$
as well as the vanishing of $H^i(Y,F^*)$ for $i \ne 0,1$. Note that
by Riemann--Roch the Euler characteristic of $F^*$ is zero, hence $h^0(F^*) = h^1(F^*)$.
Denoting this integer by $h$ we can rewrite the above sequence as
%
%Since also $h^0(F^*) = n-2$,
%we deduce that $h^1(F^*) = n-2$, so the above sequence looks like
$$
0 \to \CO_Y^h \to F^* \to F \to \CO_Y^h \to 0.
$$
Let $E$ be the image of the map $F^* \to F$. Note that $E$ is locally free
(as a kernel of an epimorphism of vector bundles). Moreover,
$c_1(E) = 0$ and $c_2(E) = n$, hence $r(E) \ge 2$.
Thus $h = n -r(E) \le n-2$.

%First, it follows easily that $h \le n$ since the rank of $F$ is $n$.
%Moreover, if $h = n$ then the image $E$ of the middle map has rank $0$,
%and is locally free (as a kernel of an epimorphism of vector bundles),
%hence $E = 0$, so $F \cong \CO_Y^n$ which contradicts $c_2(F) = n$.
%Further, if $h = n-1$ then $E$ has rank $1$ and locally free. Moreover $c_1(E) = 0$,
%hence $E = \CO_Y$. But $c_2(E) = n$ which is a contradiction.
%So, we conclude that $h \le n-2$.

Finally, if $h = n-2$ then $E$ has rank 2, is locally free, and $c_1(E) = 0$, $c_2(E) = n$.
Moreover, it is stable since $H^0(Y,E) = \Coker(H^0(Y,F^*) \to H^0(Y,F^*)) = 0$,
and $H^1(Y,E(-1)) = 0$ since both $F(-1)$ and $\CO_Y(-1)$ are acyclic.
\end{proof}

%\section{Monadic description}

%Note that by construction of the acyclic extension of an instanton, it is an object of $\CB_Y$.
%In some cases we have an explicit description of the category $\CB_Y$, it is instructive to see
%what kind of objects we obtain there. But before this we observe the following

%So, it is important to rewrite the antiautoequivalence $\SD$ in terms of the category $\CB_Y$.

%Before going into a case-by-case analysis we describe some common properties of $\CB_Y$ and $\SD$.

\subsection{Ideals of lines}

Recall that a line on a Fano threefold $Y$ is a rational curve on $Y$ of degree 1.

\begin{proposition}\label{dil}
For any Fano threefold $Y$ of index $2$ and any line $L \subset Y$ the ideal sheaf $I_L$ is contained in $\CB_Y$.
Moreover, it is fixed by $\SD$
$$
\SD(I_L) \cong I_L.
$$
\end{proposition}
\begin{proof}
From the exact sequence
\begin{equation}\label{il}
0 \to I_L \to \CO_Y \to \CO_L \to 0
\end{equation}
it follows easily that $H^\bullet(Y,I_L) = H^\bullet(Y,I_L(-1)) = 0$, so $I_L \in \CB_Y$.
Further, applying $\RCHom(-,\CO_Y)$ and taking into account that
$$
\RCHom(\CO_L,\CO_Y) \cong \CO_L[-2]
$$
by Grothendieck duality (since $\omega_{L/Y} = \omega_L\otimes\omega_{Y|L}^{-1} = \CO_L(-2)\otimes\CO_L(2) = \CO_L$),
we obtain a triangle
\begin{equation}\label{ildual}
\CO_Y \to \RCHom(I_L,\CO_Y) \to \CO_L[-1].
\end{equation} 
Since $\LL_\CO(\CO_Y) = 0$ we conclude that
$$
\SD(I_L) = \LL_\CO(\RCHom(I_L,\CO_Y)) = \LL_\CO(\CO_L[-1]) = \Cone(\CO_Y[-1] \to \CO_L[-1]) = I_L
$$
hence the claim.
\end{proof}

\begin{remark}
In fact one can show that the isomorphism $\SD(I_L) \cong I_L$ is skew-symmetric in the sense
of Proposition~\ref{dte}. However we will not need this fact, so we skip the proof.
\end{remark}
%{\tt QUESTION. What about (skew)-symmetricity of this isomorphism? Probably, one should find out whether the isomorphism
%$(\CO_L[-1])^\vee \cong \CO_L[-1]$ is symmetric or skew-symmetric, and then use the same argument as for the instantons.}

As we will see below the ideals of lines give a connection between the geometric
and categorical properties of lines. However, sometimes it is more convenient
to use the (twisted and shifted) dual objects. We denote
\begin{equation}\label{defjl}
J_L := \RCHom(I_L,\CO_Y(-1))[1] \in \D^b(Y).
\end{equation}

\begin{lemma}
We have a distinguished triangle
\begin{equation}\label{jle}
\CO_Y(-1)[1] \to J_L \to \CO_L(-1).
\end{equation}
Moreover, $J_L \in \CB_Y$.
\end{lemma}
\begin{proof}
The triangle is obtained from~\eqref{ildual} by a shift and a twist.
Since both $\CO_Y(-1)$ and $\CO_L(-1)$ are acyclic we conclude that $J_L \in \CO_Y^\perp$.
On the other hand
\begin{multline*}
\RHom(\CO_Y(1),J_L) =
\RHom(\CO_Y(1),\RCHom(I_L,\CO_Y(-1)[1])) =
\RHom(\CO_Y(1)\otimes I_L,\CO_Y(-1)[1]) = \\ =
%H^\bullet(Y,\RCHom(I_L,\CO_Y(-2)[1])) = \\ =
\RHom(I_L,\CO_Y(-2)[1]) \cong
\RHom(\CO_Y,I_L[2])^* = 0
\end{multline*}
(we used the Serre duality in the last isomorphism) hence $J_L \in \CB_Y$.
\end{proof}

\begin{remark}\label{jlcone}
One can check that the object $J_L$ is isomorphic to a cone of the unique nontrivial
morphism $\CO_L(-1)[-1] \to \CO_{Y}(-1)[1]$. Indeed, it is a cone of such a morphism
just by~\eqref{jle}, and the morphism is nontrivial since otherwise we would have
$J_L \cong \CO_{Y}(-1)[1] \oplus \CO_L(-1)$ and thus $J_L$ would not be orthogonal to $\CO_{Y}(1)$.
Finally, to check that the morphism is unique we note that it is obtained by the antiautoequivalence
$\RCHom(-,\CO_{Y}(-1)[1])$ from the morphism $\CO_Y \to \CO_L$. The later morphism is evidently unique
hence the claim.
\end{remark}

%we conclude that $\RCHom(I_L,\CO_Y)$ has cohomology $\CO_Y$ in degree $0$ and $\CO_L$ in degree $1$.
%It follows that $\RHom(I_L,\CO_Y) = \kk \oplus \kk[-1]$, hence the cohomology sequence of the distinguished triangle
%$$
%\RHom(I_L,\CO_Y)\otimes\CO_Y \to \RCHom(I_L,\CO_Y) \to \SD(I_L)
%$$
%looks like
%$$
%\dots \to 0 \to \CH^{-1}(\SD(I_L)) \to \CO_Y \to \CO_Y \to \CH^0(\SD(I_L)) \to \CO_Y \to \CO_L \to \CH^2(\SD(I_L)) \to 0 \to \dots
%$$
%and the maps $\CO_Y \to \CO_Y$ and $\CO_Y \to \CO_L$ induce isomorphisms on cohomology.
%It follows that the first map is an isomorphism, while the second is an epimorphism.
%Consequently $\CH^0(\SD(I_L)) \cong I_L$, while all the other cohomology sheaves are zero.
%Hence the claim.
%\end{proof}

\subsection{Jumping lines}

A line $L \subset Y$ is a {\sf jumping line}\/ for an instanton $E$ if $E_{|L} \cong \CO_L(i) \oplus \CO_L(-i)$ with $i > 0$.
More precisely we will say in this case that $L$ is an {\sf $i$-jumping line}. By analogy with the case of instantons on $\PP^3$
it is very tempting to state the following

\begin{conjecture}\label{jl-conj}
For any instanton $E$ on $Y$ a generic line is not jumping.
\end{conjecture}

The standard approach \cite{OSS} to this Conjecture does not work because the map from the universal line to~$Y$ has disconnected fibers
(as we have seen in Lemma~\ref{qfib} the map is finite). We will show in sections~\ref{sy5} and~\ref{sy4} that this Conjecture
is related to some well known geometric questions.

Assume that $E$ is an instanton such that generic line is not jumping for $E$. Let $D_E \subset F(Y)$
be the subscheme parameterizing jumping lines of $E$ and write $i:D_E \to F(Y)$ for the embedding.
Also recall the notation introduced in section~\ref{ss-fy}.
The following result is an analogue of the Grauert--M\"ulich Theorem.

% . Recall that $D_L \subset F(Y)$ is the curve
% parameterizing lines intersecting given line $L$.

\begin{theorem}
If $E$ is an instanton on $Y$ of charge $n$ such that generic line is not jumping for~$E$ then
$$
D_E \sim nD_L.
$$
Further, there is a coherent sheaf $\CL_E$ on $D_E$ such that
$$
Rp_*q^*E(-1) \cong i_*\CL_E[-1].
$$
The sheaf $\CL_E$ is invertible on the open subset of $D_E$ parameterizing $1$-jumping lines,
and has the property
$$
\CL_E \cong \RCHom(\CL_E,\CO_{D_E}((n-d)D_L).
$$
In particular, if $E$ has no $2$-jumping lines then $\CL_E$ is a line bundle such that $\CL_E^2 \cong \CO_{D_E}((n-d)D_L)$.
\end{theorem}
\begin{proof}
Consider the object $\CF := Rp_*q^*E(-1) \in \D^b(F(Y))$. If $x$ is a point of $F(Y)$
such that the corresponding line $L_x$ on $Y$ is not a jumping line then
$H^\bullet(L_x,E(-1)_{|L_x}) = 0$, hence $\CF$ is supported on the subscheme $D_E$.
Further, if $L_x$ is a 1-jumping line then $H^\bullet(L_x,E(-1)_{|L_x}) = \kk \oplus \kk[-1]$,
which means that $\CF$ is a rank 1 sheaf on $D_E$ shifted by $-1$.
Thus
$$
D_E = - c_1(\CF) = - c_1(Rp_*q^*E(-1)).
$$
Note that by Grothendieck--Riemann--Roch the first Chern class of $Rp_*q^*(E(-1))$
does not depend on $E$ itself, it depends only on the Chern character of $E$. In particular, 
to compute the rational equivalence class of $D_E$ we can replace $E$ by any sheaf 
with the same Chern character. The most convenient choice is to take
$$
E' = \Ker \left(\CO_Y^{\oplus 2} \to \oplus_{i=1}^n \CO_{L_i}\right),
$$
where $L_1,\dots,L_n$ is a generic $n$-tuple of lines.
It is clear that $Rp_*q^*\CO_Y(-1) = 0$, hence we have
$Rp_*q^*E'(-1) \cong \oplus Rp_*q^*\CO_{L_i}(-1)[-1]$.
It remains to check that $c_1(Rp_*q^*\CO_{L_i}(-1)) = D_{L_i}$.

Indeed, let $L_i$ be the line corresponding to a point $x_i \in F(Y)$.
As $L_i$ is generic, we may assume that the map $q$ is flat over $L_i$,
so $q^*\CO_{L_i} = \CO_{q^{-1}(L_i)}$. But it is clear that
$$
q^{-1}(L_i) = p^{-1}(x_i) \cup \tilde D_{L_i},
$$
where $\tilde D_{L_i}$ is a section of the map $p$ over $D_{L_i}$
(the points of $\tilde D_{L_i}$ are the pairs $(y,x) \in Y\times F(Y)$
such that $x \in D_L$ and $y$ is the unique point of intersection
of the line $L_x$ with $L_i$). Thus we have an exact sequence
$$
0 \to G_{\tilde D_{L_i}} \to q^*\CO_{L_i} \to \CO_{p^{-1}(x_i)} \to 0,
$$
where $G_{\tilde D_{L_i}}$ is the sheaf of ideals of the scheme-theoretical intersection
$p^{-1}(x_i) \cap \tilde D_{L_i}$ on $\tilde D_{L_i}$. In particular, it is a sheaf of rank $1$
on $\tilde D_{L_i}$.
Tensoring the above sequence by $q^*\CO_Y(-1)$ and taking into account that
$Rp_*(\CO_{p^{-1}(x_i)}\otimes q^*\CO_Y(-1)) = H^\bullet(p^{-1}(x),\CO_{p^{-1}(x)}(-1)) \otimes \CO_x = 0$
since $p^{-1}(x) = \PP^1$, we conclude that
$$
Rp_*q^*(\CO_{L_i}(-1)) =
Rp_*(q^*\CO_{L_i}\otimes q^*\CO_Y(-1)) =
Rp_*(G_{\tilde D_{L_i}}\otimes q^*\CO_Y(-1)).
$$
Since the restriction of the map $p$ to $\tilde D_{L_i}$ is an isomorphism onto $D_{L_i}$ we conclude
that $Rp_*q^*(\CO_{L_i}(-1))$ is a rank $1$ sheaf on $D_{L_i}$. 
Hence its first Chern class indeed equals $D_{L_i}$.

For the second claim we have to check that $\CF$ is a coherent sheaf shifted by $-1$.
Since the map $p$ has relative dimension 1, the object $\CF$ can have cohomology only in degree $0$ and $1$.
Thus we have to check that the cohomology in degree $0$ vanishes. Indeed, let $\CF^0$ denote the cohomology of $\CF$
in degree 0 and $\CF^1$ the cohomology in degree 1. Then we have a distinguished triangle
$$
\CF^0 \to \CF \to \CF^1[-1].
$$
Applying the Grothendieck duality and taking into account that $\omega_{Z/F(Y)} = p^*\CO_{F(Y)}(dD_L)\otimes q^*\CO_Y(-2)$
by Corollary~\ref{kzf}, we have
\begin{multline*}
\RCHom(\CF,\CO_{F(Y)}) =
\RCHom(Rp_*q^*(E(-1)),\CO_{F(Y)}) \cong
Rp_*\RCHom(q^*(E(-1)),p^!\CO_{F(Y)}) \cong \\ \cong
Rp_*(q^*(E^*(1))\otimes \omega_{Z/F(Y)}[1]) \cong
Rp_*(q^*(E^*(1))\otimes p^*\CO_{F(Y)}(dD_L)\otimes q^*\CO_Y(-2)[1]) \cong \\ \cong
Rp_*(q^*(E(-1)))\otimes p^*\CO_{F(Y)}(dD_L)[1]
\cong \CF(dD_L)[1].
\end{multline*}
%Since $\det M^* \cong \CO_{F(Y)}(dD_L)$ We conclude that
%\begin{equation}
%\RCHom(\CF,\CO_{F(Y)}) \cong \CF(dD_L)[1].
%\end{equation}
On the other hand, applying duality to the distinguished triangle for $\CF$ we obtain a triangle
$$
\RCHom(\CF^1,\CO_{F(Y)})[1] \to \CF(dD_L)[1] \to \RCHom(\CF^0,\CO_{F(Y)}).
$$
Note that since both $\CF^0$ and $\CF^1$ are supported on a closed subscheme of $F(Y)$,
their derived duals a concentrated in degrees higher than 1. Hence the first and the third term
of the triangle are concentrated in nonnegative degrees. It follows that the cohomology of $\CF(dD_L)[1]$
in degree $-1$, which is nothing but $\CF^0(dD_L)$, vanishes. Thus $\CF^0 = 0$ and $\CF = \CF^1[-1]$.
Moreover, since $\CF_1$ is supported on the curve $D_E$ we can write
$\CF = i_*\CL_E[-1]$, this being a definition of the coherent sheaf $\CL_E$.
We have already seen the sheaf $\CL_E$ is of rank 1 at any point of $D_E$ corresponding
to a 1-jumping line.

Finally, recall that $\RCHom(\CF,\CO_{F(Y)}) \cong \CF(dD_L)[1]$. Substituting here $\CF = i_*\CL_E[-1]$
and using the Grothendieck duality we deduce
%
%To check the final statement note that if there is no 2-jumping lines then
%the fiber of the sheaf $\CF^1 = R^1p_*q^*E(-1)$ at any point $x \in D_E$ is isomorphic to
%$H^1(L_x,E(-1)_{|L_x}) \cong H^1(L_x,\CO \oplus \CO(-2)) = \kk$, hence $\CF^1 \cong i_*\CL_E$,
%and so $\CF \cong i_*\CL_E[-1]$, where $\CL_E$ is a line bundle on $D_E$ and $i:D_E \to F(Y)$ is the embedding.
%To finish, note that again by Grothendieck duality we have
\begin{multline*}
i_*\CL_E(dD_L) \cong
%\CF(dD_L)[1] \cong \RCHom(\CF,\CO_{F(Y)}) =
\RCHom(i_*\CL_E[-1],\CO_{F(Y)}) \cong
i_*\RCHom(\CL_E[-1],i^!\CO_{F(Y)}) \cong \\ \cong
i_*\RCHom(\CL_E,\omega_{D_E/F(Y)}) \cong
i_*\RCHom(\CL_E,\CO_{D_E}(D_E)) \cong
i_*\RCHom(\CL_E,\CO_{D_E}(nD_L))
%
%i_*(\CL_E^*[1]\otimes \omega_{D_E/F(Y)}[-1]) \cong
%i_*(\CL_E^*\otimes \CO_{D_E}(D_E)).
\end{multline*}
which gives the required property of $\CL_E$. Finally, if there are no 2-jumping lines
and so $\CL_E$ is a line bundle this is evidently equivalent to $\CL_E^2 \cong \CO_{D_E}((n-d)D_L)$.
\end{proof}

Now we can state the following

\begin{conjecture}\label{jlines}
The curve of jumping lines $D_E$ together with the line bundle $\CL_E$ determines the instanton.
\end{conjecture}

Again, the standard reconstruction procedure \cite{OSS} does not work here since the lines corresponding
to points of $D_E$ do not sweep $Y$ (they sweep a certain surface), so it is not clear a priori
how one could produce the bundle $E$ out of this surface. We will see however that for Fano threefolds
of degree~5 and~4 the Conjecture is true.

\subsection{Jumping lines in terms of $\CB_Y$}

It turns out that the curve of jumping lines can be described in the intrinsic terms of the category $\CB_Y$.
This description will be useful later. To make a statement recall that for each line $L$ we have defined
an object $J_L = \RCHom(I_L,\CO_Y(-1))[1] \in \D^b(Y)$. This can be used to construct a universal family
of objects $J_L$.

Indeed, first note that the universal family of ideal sheaves $I_L$ is the ideal sheaf $I_Z$ on $Y\times F(Y)$,
where $Z$ is the universal line. Denote the embedding of $Z$ into $Y \times F(Y)$ by $\zeta$. Now consider
$$
\CJ = \RCHom(I_Z,q_1^*\CO_Y(-1)\otimes p_1^*\CO_{F(Y)}(-dD_L)[1]),
$$
where $p_1$ and $q_1$ are the projections from $Y\times F(Y)$ to $F(Y)$ and $Y$ respectively.
Applying the functor $\RCHom(-,q_1^*\CO_Y(-1)[1])$ to the exact sequence $0 \to I_Z \to \CO_{Y\times F(Y)} \to \CO_Z \to 0$
and taking into account the fact that by Grothendieck duality we have
\begin{multline*}
\RCHom(\CO_Z,q_1^*\CO_Y(-1)\otimes p_1^*\CO_{F(Y)}(-dD_L)[1]) \cong
\zeta_*\zeta^!(q_1^*\CO_Y(-1)\otimes p_1^*\CO_{F(Y)}(-dD_L))[1] \cong \\ \cong
\zeta_*(q^*\CO_Y(-1)\otimes p^*\CO_{F(Y)}(-dD_L)\otimes\omega_{Z/Y\times F(Y)}[-1]) \cong
\zeta_*q^*\CO_Y(-1)[-1] \cong \CO_Z(-1)[-1],
\end{multline*}
we deduce that $\CJ$ fits into the following distinguished triangle
\begin{equation}\label{cj}
q_1^*\CO_Y(-1)\otimes p_1^*\CO_{F(Y)}(-dD_L)[1] \to \CJ \to \CO_Z(-1).
\end{equation}

\begin{proposition}\label{tejl}
Let $\TE$ be the acyclic extension of an instanton $E$.
A line $L$ on $Y$ is a jumping line for $E$ if and only if $\Hom(\TE,J_L) \ne 0$.
Moreover, we have
$$
Rp_*q^*E(-1) \cong Rp_{1*}\RCHom(q_1^*\TE,\CJ).
$$
In particular, if generic line is not jumping for $E$ then
$Rp_{1*}\RCHom(q_1^*\TE,\CJ) \cong i_*\CL_E[-1]$.
\end{proposition}
\begin{proof}
First, $J_L \in \CB_Y \subset \CO_Y^\perp$, hence $\Ext^\bullet(\TE,J_L) = \Ext^\bullet(E,J_L)$.
Further,
$$
\Ext^\bullet(E,\CO_Y(-1)) = H^\bullet(Y,E^*(-1)) = H^\bullet(Y,E(-1)) = 0
$$
by self-duality of $E$, hence $\Ext^\bullet(E,J_L) = \Ext^\bullet(E,\CO_L(-1))$. Finally,
using again the self-duality of $E$ we see that
$$
\Ext^\bullet(E,\CO_L(-1)) =
H^\bullet(Y,E^*\otimes\CO_L(-1)) =
H^\bullet(Y,E\otimes\CO_L(-1)) =
H^\bullet(L,E_{|L}(-1)).
$$
Combining all this we see that for non-jumping line $L$ we have $\Ext^\bullet(\TE,J_L) = 0$,
while for an $i$-jumping line $L$ we have $\dim\Hom(\TE,J_L) = \dim\Ext^1(\TE,J_L) = i$.

For the second statement we apply the functor $Rp_{1*}\RCHom(q_1^*\TE,-)$ to the triangle~\eqref{cj}.
Note that
\begin{multline*}
Rp_{1*}\RCHom(q_1^*\TE,q_1^*\CO_Y(-1)\otimes p_1^*\CO_{F(Y)}(-dD_L)) \cong
Rp_{1*}(q_1^*\TE^*(-1)\otimes p_1^*\CO_{F(Y)}(-dD_L)) \cong \\ \cong
H^\bullet(Y,\TE^*(-1))\otimes \CO_{F(Y)}(-dD_L) = 0
\end{multline*}
since $\TE^*$ is an extension of $E^* \cong E$ by $\CO_Y^{n-2}$ and both bundles
are in $\CO_Y(1)^\perp$. On the other hand,
\begin{equation*}
Rp_{1*}\RCHom(q_1^*\TE,\CO_Z(-1)) \cong
Rp_*q^*(\TE^*(-1))
\end{equation*}
% where in the RHS the maps $p$ and $q$ are considered as maps from $Z$. 
Again, since $\TE^*$ is an extension
of $E$ by $\CO_Y^{n-2}$ and $Rp_{1*}q_1^*(\CO_Y(-1)) = 0$ by base change
% H^\bullet(Y,\CO_Y(-1)) \otimes\CO_{F(Y)} = 0$, 
we conclude that $Rp_{1*}q_1^*(\TE^*(-1)) \cong Rp_*q^*(E(-1))$.
Combining all this we deduce the required isomorphism. 
\end{proof}

The same trick can be used for the description of the divisor of intersecting lines in $F(Y)\times F(Y)$
and for the curve $D_L \subset F(Y)$ as well.

\begin{lemma}\label{llpint}
Two distinct lines $L$ and $L'$ intersect if and only if $\Hom(I_L,J_{L'}) \ne 0$.
\end{lemma}
\begin{proof}
Since $\Ext^\bullet(\CO_Y,J_{L'}) = 0$ we have $\Ext^\bullet(I_L,J_L) = \Ext^{\bullet-1}(\CO_L,J_{L'})$.
Similarly, by Serre duality we have $\Ext^\bullet(\CO_L,\CO_Y(-1)) \cong \Ext^\bullet(\CO_Y(-1),\CO_L(-2)[3])^* \cong H^\bullet(L,\CO_L(-1)[3])^* = 0$,
whence $\Ext^\bullet(\CO_L,J_{L'}) \cong \Ext^\bullet(\CO_L,\CO_{L'})$. On the other hand, if lines $L$ and $L'$ do not intersect then this is zero.
If they intersect in a point then $\Ext^i(\CO_L,\CO_{L'}) = \kk$ for $i = 1$ and $i = 2$. Combining with the above isomorphisms we conclude that
$$
\Ext^i(I_L,J_{L'}) = \begin{cases} \kk, & \text{if $L$ intersects $L'$ and $i = 0$, $1$}\\ 0, & \text{otherwise} \end{cases}
$$
which proves the Lemma.
\end{proof}

\section{Instantons on Fano threefolds of degree $5$}\label{sy5}

In this section we consider in detail the case of the Fano threefold $Y_5$ of index 2 and degree 5.
We start with a short reminder on the geometry and derived category of $Y_5$.

\subsection{Derived category}

Recall that $Y_5$ is a linear section of codimension 3 of $\Gr(2,5)$.
Denote by $V$ the vector space of dimension 5 and by $A \subset \Lambda^2V^*$
a generic vector subspace of dimension 3 (the group $\SL(V)$ acts with an open orbit
on the Grassmannian $\Gr(3,\Lambda^2V^*)$ and any $A$ from the open orbit gives the same linear section). 
Denote also by $\CU$ the restriction of the tautological rank 2 subbundle from $\Gr(2,V)$ to~$Y_5$ and let
$$
\CU^\perp = \Ker(V^*\otimes\CO_Y \to \CU^*).
$$
Recall that by~\cite{Or91} the category $\D^b(Y_5)$ is generated by an exceptional collection.
For our purposes the most convenient choice of the collection is
\begin{equation}\label{dby5}
\D^b(Y_5) = \langle \CU, \CU^\perp, \CO_{Y_5}, \CO_{Y_5}(1) \rangle.
\end{equation}

It gives the following descriptions of the category $\CB_{Y_5}$.

\begin{lemma}\label{by5}
The category $\CB_{Y_5}$ is generated by either of the following two exceptional pairs
$$
\CB_{Y_5} = \langle \CU, \CU^\perp \rangle = \langle (V/\CU)(-1),\CU \rangle.
$$
Moreover, we have canonical isomorphisms
$$
\Ext^\bullet(\CU,\CU^\perp) =
%A,
%\qquad
\Ext^\bullet((V/\CU)(-1),\CU) = A.
$$
\end{lemma}
\begin{proof}
The first decomposition follows immediately from the definition of $\CB_{Y_5}$ and~\eqref{dby5}.
To get the second, we apply to $\D^b(Y_5)$ the antiautoequivalence $F \mapsto \RCHom(F,\CO_{Y_5}(-1))$.
Since $(\CU^\perp)^* = V/\CU$ and $\CU^*(-1) \cong \CU$, we see that it takes~\eqref{dby5} to
$$
\D^b(Y_5) = \langle \CO_{Y_5}(-2),\CO_{Y_5}(-1),(V/\CU)(-1),\CU \rangle.
$$
Finally, by Serre duality we have
$$
\CB_{Y_5} =
\langle \CO_{Y_5}, \CO_{Y_5}(1) \rangle^\perp =
{}^\perp\langle \CO_{Y_5}(-2), \CO_{Y_5}(-1) \rangle,
$$
which gives the second decomposition of $\CB_{Y_5}$.

%note that the right mutations through $\CO_{Y_5}$ of $\CU^\perp$ and $\CU$
%are given by the tautological exact sequences
%$$
%0 \to \CU^\perp \to V^*\otimes\CO_{Y_5} \to \CU^* \to 0,
%\qquad
%0 \to \CU \to V\otimes\CO_{Y_5} \to V/\CU \to 0,
%$$
%which means that $\D^b(Y_5)$ is generated by the exceptional collection $(\CO_{Y_5},V/\CU,\CU^*,\CO_{Y_5}(1))$.
%Now mutating $\CO_{Y_5}$ to the right through all other objects of this collection we obtain $\CO_{Y_5}(2)$
%(since the inverse of the Serre functor is the twist by $\CO_{Y_5}(2)$),  so we see that $\D^b(Y_5)$
%is also generated by the exceptional collection $(V/\CU,\CU^*,\CO_{Y_5}(1),\CO_{Y_5}(2))$.
%Finally, twisting by $\CO_{Y_5}(-1)$ we see that $\D^b(Y_5)$ is also generated by the exceptional
%collection $((V/\CU)(-1),\CU^*(-1),\CO_{Y_5},\CO_{Y_5}(1))$. Since $\CU^*(-1) \cong \CU$
%we deduce the second decomposition of $\CB_{Y_5}$.

For the computation of $\Ext$'s we refer to~\cite{Or91}. Here we will only explain how the evaluation
morphism
$$
\alpha:A \otimes \CU \to \CU^\perp
$$
can be described. Consider the map
$A\otimes\CU \to A\otimes V\otimes\CO_{Y_5} \xrightarrow{\ \ev\ } V^*\otimes\CO_{Y_5}$,
where $\ev$ is the evaluation of a 2-form (recall that $A$ is a subspace in $\Lambda^2V^*$) on a vector.
Its composition with the projection $V^*\otimes\CO_{Y_5} \to \CU^*$ vanishes (by definition of $Y_5$),
hence the map itself factors through the subbundle $\CU^\perp$.
%(each skew form $\omega$ on $V$ gives a map $\CU \to V\otimes\CO_{Y_5} \xrightarrow{\ \omega\ } V^*\otimes\CO_{Y_5}$
%and if $\omega \in A$ then its composition with the projection $V^*\otimes\CO_{Y_5} \to \CU^*$ vanishes,
%hence the map itself factors through $\CU^\perp$).
\end{proof}

We would like to point out the following two funny consequences of the Lemma. First, observe that it follows
that the left mutation of $\CU^\perp$ through $\CU$ is $(V/\CU)(-1)[1]$ and dually, the right mutation of $(V/\CU)(-1)$
through $\CU$ is $\CU^\perp[-1]$. In the other words, we have the following exact sequence
\begin{equation}\label{seqy5}
0 \to (V/\CU)(-1) \to A\otimes\CU \to \CU^\perp \to 0.
\end{equation}
Also note that the antiautoequivalence from the proof of Lemma~\ref{by5} takes the above exact sequence to
$$
0 \to (V/\CU)(-1) \to A^*\otimes\CU \to \CU^\perp \to 0.
$$
Since the sequence is canonical, it follows that there is an isomorphism
\begin{equation}\label{aas}
A \cong A^*,
\end{equation}
which can be easily shown to be symmetric. From now on for each vector $a \in A$ we will
denote by $a^* \in A^*$ the covector corresponding to $a$ under isomorphism~\eqref{aas}.

\subsection{The Fano scheme of lines}

It is well known that the Fano scheme of lines on $Y_5$ is $\PP^2$. We will need the following
more precise description.

\begin{lemma}
We have $F(Y_5) = \PP(A)$. Moreover, for each point $a \in \PP(A)$
%if $H_a \subset A$ the 2-dimensional subspace corresponding to $a^* \in \PP(A^*)$ then
we have an exact sequence
\begin{equation}\label{ily5}
0 \to \CU \xrightarrow{\ a\ } \CU^\perp \to I_L \to 0,
\end{equation}
and a distinguished triangle
\begin{equation}\label{jly5}
%H_a\otimes\CU \to \CU^\perp \to J_L,
%\qquad
(V/\CU)(-1) \xrightarrow{\ a\ } \CU \to J_L.
\end{equation}
\end{lemma}
\begin{proof}
Stability of $V/\CU$ and $\CU^*$ implies that the morphism $a:(V/\CU)(-1) \to \CU$ has kernel of rank $1$.
Since it is reflexive, we conclude that it is a line bundle. Again, by stability of $V/\CU$ and $\CU^*$
we know that it has degree $-1$, so the kernel is $\CO_{Y_5}(-1)$. Computing the Chern class of the cokernel
we see that it is a torsion sheaf of rank $1$ on some line $L$ on $Y_5$. Moreover, since the sheaves
$\CO_{Y_5}(-1)$, $(V/\CU)(-1)$ and $\CU$ are all acyclic, the cokernel is acyclic as well.
In particular, it has no $0$-dimensional torsion, so it is a line bundle on $L$,
which being acyclic should be isomorphic to $\CO_L(-1)$. Thus we obtain an exact sequence
$$
0 \to \CO_{Y_5}(-1) \to (V/\CU)(-1) \xrightarrow{\ a\ } \CU \to \CO_L(-1) \to 0.
$$
In other words, we see that the cone of $a:(V/\CU)(-1) \to \CU$ is quasi-isomorphic to the (shifted by $1$) cone
of a morphism $\CO_L(-1) \to \CO_{Y_5}(-1)[2]$ and as it was explained in Remark~\ref{jlcone} to justify the
triangle~\eqref{jly5} it remains to show that this morphism is nontrivial. Indeed, if the morphism were trivial
then the cone will be the direct sum of $\CO_{Y_5}(-1)[1]$ and $\CO_L(-1)$, which should imply in particular
that the surjection $\CU \to \CO_L(-1)$ splits, which of course is false as $\CU$ is torsion free.

Now to obtain the first exact triangle it is sufficient to remember that $J_L = \RCHom(I_L,\CO_{Y_5}(-1))[1]$
(just by definition). Since $\RCHom(-,\CO_{Y_5}(-1))[1]$ is an involution, we can apply it to~\eqref{jly5}.
It is easy to see that one will get precisely~\eqref{ily5}.
%
%Stability of $\CU$ and $\CU^\perp$ implies that the map $\CU \to \CU^\perp$ is injective.
%Denote the cokernel by $C$ and dualize exact sequence $0 \to \CU \to \CU^\perp \to C \to 0$.
%We will get a long exact sequence
%$$
%0 \to \CHom(C,\CO_{Y_5}) \to V/\CU \xrightarrow{\ a\ } \CU^* \to \CExt^1(C,\CO_{Y_5}) \to 0.
%$$
%
\end{proof}

\begin{remark}
Alternatively, the object $J_L$ can be written as the cone of a morphism $a^\perp\otimes\CU \to \CU^\perp$,
where $a^\perp \subset A$ is the orthogonal complement of $a \in A$. It follows from Lemma~\ref{llpint}
that lines $L$ and $L'$ intersect if and only if the corresponding vectors $a,a' \in A$ are orthogonal.
Thus, the divisor $D_L$ is the line on $\PP(A)$ orthogonal to $a$ with respect to the quadratic form on $A$
corresponding to the isomorphism~\eqref{aas}.
\end{remark}

\subsection{The action of the antiautoequivalence}

Let us describe the antiautoequivalence $\SD$. For this it suffices to understand how it acts on the bundles $\CU$ and $\CU^\perp$.

\begin{lemma}\label{sd5}
We have $\SD(\CU) = \CU^\perp[1]$, $\SD(\CU^\perp) = \CU[1]$.
Moreover, the morphism
$$
\SD[-1]:A = \Hom(\CU,\CU^\perp) \to \Hom(\SD[-1](\CU^\perp),\SD[-1](\CU)) = \Hom(\CU,\CU^\perp) = A
$$
is $-1$.
%
% $\delta_\CU = 1$ and $\delta_{\CU^\perp} = 1$.
\end{lemma}
\begin{proof}
Indeed, we have $\RCHom(\CU,\CO_{Y_5}) = \CU^*$ and $\LL_\CO(\CU^*) = \Cone(V^*\otimes\CO_{Y_5} \to \CU^*) = \CU^\perp[1]$.
Similarly, $\RCHom(\CU^\perp,\CO_{Y_5}) = V/\CU$ and $\LL_\CO(V/\CU) = \Cone(V\otimes\CO_{Y_5} \to V/\CU) = \CU[1]$.

To check the second part take any $a \in A$ and the corresponding morphism $\alpha_a:\CU \to \CU^\perp$.
By definition $\alpha_a$ factors as $\CU \xrightarrow{\ a\ } A\otimes\CU \xrightarrow{\ \alpha\ } \CU^\perp$.
Dualizing we obtain the morphism $\alpha_a^*$ which factorizes as
$V/\CU \xrightarrow{\ \alpha^*\ } A^*\otimes\CU^* \xrightarrow{\ a\ } \CU^*$. Note that it also factorizes as
$V/\CU \xrightarrow{\ -a\ } A\otimes (V/\CU) \xrightarrow{\ \alpha^*\ } \CU^*$.
It follows that after the mutation $\LL_\CO$ (and a shift) we obtain a map $\CU \to \CU^\perp$
which factorizes as $\CU \xrightarrow{\ -a\ } A\otimes\CU \xrightarrow{\ \alpha\ } \CU^\perp$,
hence coincides with $-\alpha_a$.
%%
%
%To compute $\delta_\CU$ consider the distinguished triangle
%$$
%V^*\otimes\CO_{Y_5} \to \CU^* \to \CU^\perp[1].
%$$
%After dualization we have
%$$
%V/\CU[-1] \to \CU \to  V\otimes\CO_{Y_5}.
%$$
%Applying $\LL_\CO$ we see that $\delta_\CU = 1$. Analogously $\delta_{\CU^\perp} = 1$.
\end{proof}

\subsection{The monadic description}

As Lemma~\ref{by5} shows we have an equivalence
$$
\CB_{Y_5} \cong \D^b(\SQ_A),
$$
where $\D^b(\SQ_A)$ is the derived category of finite dimensional representations of the quiver
with 2 vertices and the space of arrows from the first vertex to the second given by $A$
$$
\SQ_A = \bullet \xrightarrow{\ A\ } \bullet.
%\xymatrix@1{\bullet \ar[r] \ar@<.7ex>[r] \ar@<-.7ex>[r] & \bullet}.
$$
The equivalence is given by
$$
\Phi_5:\D^b(\SQ_A) \to \CB_{Y_5},
\qquad
(M_1^\bullet,M_2^\bullet,m) \mapsto
%(\xymatrix@1{M_1^\bullet \ar[r] \ar@<.7ex>[r] \ar@<-.7ex>[r] & M_2^\bullet}) \mapsto
\Cone(M_1^\bullet\otimes\CU \xrightarrow{\ m\ } M_2^\bullet\otimes\CU^\perp).
$$
The inverse equivalence $\Phi_5^{-1}:\CB_{Y_5} \to \D^b(\SQ_A)$ takes any $F \in \CB_{Y_5}$ to the representation $(M_1^\bullet,M_2^\bullet)$ with
$$
M_2^\bullet = \Ext^\bullet(\CU^\perp,F),
\qquad
M_1^\bullet = \Ext^\bullet(F,\CU[1])^*.
$$
To get a monadic description of an instanton we just apply $\Phi_5^{-1}$ to its acyclic extension.

\begin{lemma}
Let $F$ be a semistable vector bundle of rank $n$ with $c_1(F) = 0$ such that $F \in \CB_{Y_5}$.
Then $\Ext^\bullet(F,\CU) = \kk^n[-1]$.
\end{lemma}
\begin{proof}
First, note that $\CU \cong \CU^*(-1)$ (since $\CU$ has rank 2 and $\det\CU \cong \CO_{Y_5}(-1)$), hence we have
the following exact triple
$$
0 \to \CU^\perp(-1) \to V^*\otimes\CO_{Y_5}(-1) \to \CU \to 0
$$
(this is just the exact triple defining $\CU^\perp$ twisted by $-1$).
By Serre duality we have $\Ext^i(F,\CO_{Y_5}(-1)) = H^{3-i}(Y_5,F(-1))^* = 0$, so it follows that
\begin{equation}\label{uup}
\Ext^\bullet(F,\CU) \cong \Ext^{\bullet + 1}(F,\CU^\perp(-1)).
\end{equation}
Now note that $\mu(\CU) = -1/2$,  $\mu(F) = 0$. Therefore by stability of $F$ and $\CU$ we have
$$
\Hom(F,\CU) = 0,
\qquad\text{and}\qquad
\Hom(\CU,F(-2)) = 0.
$$
On the other hand, $\mu(\CU^\perp(-1)) = - 4/3$ and $\mu(F(-2)) = -2$, hence
by stability of $F$ and $\CU^\perp$ we have
$$
\Hom(\CU^\perp(-1),F(-2)) = 0.
$$
By Serre duality it follows that $\Ext^3(F,\CU) = 0$, $\Ext^3(F,\CU^\perp(-1)) = 0$.
Combining this with~\eqref{uup} we see that $\Ext^i(F,\CU) = 0$ unless $i = 1$.
Computing the Euler characteristic with Riemann--Roch (recall that by Remark~\ref{k0b}
we have $c_2(F) = n$ and $c_3(F) = 0$) we conclude that
$$
\Ext^\bullet(F,\CU) = \kk^n[-1].
$$
which proves the Lemma.
\end{proof}

%Denote
%$$
%A = \Hom(\CU,\CU^\perp)
%$$
%(this is a 3-dimensional vector space).
Let $H$ be a fixed vector space of dimension $n$.

\begin{proposition}\label{ftog}
Let $F$ be a semistable vector bundle of rank $n$ with $c_1(F) = 0$ such that $F \in \CB_{Y_5}$.
Choose an isomorphism $H \cong \Ext^1(F,\CU)$. If $\SD(F) \cong F$ then there is an exact sequence
$$
0 \to H\otimes\CU \xrightarrow{\ \gamma_F\ } H^*\otimes \CU^\perp \xrightarrow{\quad\ } F \to 0.
$$
If the isomorphism $\phi_F:\SD(F) \to F$ is skew-symmetric then the morphism $\gamma_F$ is given
by a symmetric in~$H$ tensor in $A\otimes H^*\otimes H^* = \Hom(H\otimes\CU,H^*\otimes\CU^\perp)$.
\end{proposition}
\begin{proof}
Consider the universal extension
\begin{equation}\label{he}
0 \to H\otimes \CU \to F' \to F \to 0.
\end{equation}
It follows that $\Ext^\bullet(F',\CU) = 0$. On the other hand,
$\Ext^\bullet(\CO_{Y_5},F') = \Ext^\bullet(\CO_{Y_5}(1),F') = 0$ since this is true
both for $F$ and $\CU$. Hence looking at exceptional collection~\eqref{dby5}
we see that $F' \in \langle \CU^\perp \rangle$, hence $F'$ is a direct sum of shifts of $\CU^\perp$.
On the other hand, from~\eqref{he} we see that $F'$ is a vector bundle of rank $2n + n = 3n$.
Hence $F' \cong (\CU^\perp)^n$. In other words, we have shown that there is an exact sequence
$$
\xymatrix@1{
0 \ar[r] &  H\otimes\CU \ar[rr]^-{\gamma_F} && H'\otimes\CU^\perp \ar[rr] && F \ar[r] & 0,
}
$$
where $H'$ is another vector space of dimension $n$. Now it is time to use the self-duality of $F$.
Applying $\SD$ and taking into account Lemma~\ref{sd5} we obtain another exact sequence
$$
\xymatrix@1{
0 \ar[r] &  (H')^*\otimes\CU \ar[rr]^-{-\gamma^T_F} && H^*\otimes\CU^\perp \ar[rr] && \SD(F) \ar[r] & 0,
}
$$
Both sequences come from a decomposition of an object of the category $\CB_Y$ with respect to the exceptional
collection $(\CU,\CU^\perp)$, hence the map $\phi_F:\SD(F) \to F$ induces a unique isomorphism of these exact sequences, that is
a pair of isomorphisms $h:H^* \to H'$, $h':(H')^* \to H$ such that the following diagram commutes
$$
\xymatrix{
0 \ar[r] &  (H')^*\otimes\CU \ar[rr]^-{-\gamma_F^T} \ar[d]^{h'} && H^*\otimes\CU^\perp \ar[rr] \ar[d]^h && \SD(F) \ar[r] \ar[d]^{\phi_F}& 0\\
0 \ar[r] &  H\otimes\CU \ar[rr]^-{\gamma_F} && H'\otimes\CU^\perp \ar[rr] && F \ar[r] & 0
}
$$
Applying the duality $\SD$ once again we obtain yet another commutative diagram
$$
\xymatrix{
0 \ar[r] &  (H')^*\otimes\CU \ar[rr]^-{-\gamma_F^T} \ar[d]^{h^T} && H^*\otimes\CU^\perp \ar[rr] \ar[d]^{(h')^T} && \SD(F) \ar[r] \ar[d]^{\SD(\phi_F)}& 0\\
0 \ar[r] &  H\otimes\CU \ar[rr]^-{\gamma_F} && H'\otimes\CU^\perp \ar[rr] && F \ar[r] & 0
}
$$
Since $\SD(\phi_F) = -\phi_F$ we conclude that $h' = -h^T$. Identifying $H'$ with $H^*$ via $h$
we see from the first diagram that $-\gamma_F = -\gamma_F^T$, so $\gamma_F^T = \gamma_F$, that is $\gamma_F$ is symmetric.
\end{proof}

%As a corollary we obtain a description of the moduli space of instantons.

For each $\gamma \in A\otimes S^2H^*$ consider the induced map $m_\gamma:H \to H^*\otimes A$.
% can be thought of as a representation of the quiver $\SQ_A$ on vector spaces $H$ and $H^*$.
Consider also the composition
$$
\gamma':H\otimes\CU \xrightarrow{\ m_\gamma\otimes\id_\CU\ } H^*\otimes A\otimes\CU \xrightarrow{\ \id_{H^*}\otimes\alpha\ } H^*\otimes\CU^\perp
$$
and
$$
\hat\gamma: H\otimes V \xrightarrow{m_\gamma\otimes\id_V} H^*\otimes A\otimes V \xrightarrow{\id_{H^*}\otimes\ev} H^*\otimes V^*.
$$

%
%
%Consider the following commutative diagram
%$$
%\xymatrix{
%H\otimes\CU \ar[rr]^-{\mu_\gamma\otimes\id_\CU} \ar[d] && H^*\otimes A\otimes\CU \ar[d] \ar@{..>}[rr] && H^*\otimes\CU^\perp \ar[d] \\
%H\otimes V\otimes\CO_{Y_5} \ar[rr]^-{\mu_\gamma\otimes\id_V} && H^*\otimes A\otimes V\otimes\CO_{Y_5} \ar[rr]^-{\id_{H^*}\otimes\mu_\gamma} && H^*\otimes V^*\otimes\CO_{Y_5} \ar[d] \\
%&&&& H^*\otimes \CU^*
%}
%$$
%Note that the composition of the maps from $H^*\otimes A\otimes \CU$ to $H^*\otimes\CU^*$ vanishes on $Y_5$, hence
%there is a map $H^*\otimes A\otimes \CU \to H^*\otimes\CU^\perp$ (denoted above by the dotted arrow)
%preserving commutativity of the diagram. Composing it with the map $\gamma$ we obtain a map
%$H\otimes\CU \to H^*\otimes\CU^\perp$ which we also denote by $\gamma$.
%
%
%
%consider the map
%$\hat\gamma:H\otimes V \to V^*\otimes H^*$ defined as the following composition
%$$
%\hat\gamma:\xymatrix@1{H\otimes V \ar[rr]^-{\gamma\otimes\id_V} && H^*\otimes A\otimes V \ar[rr]^-{\id_{H^*}\otimes\mu} && H^*\otimes V^*}.
%$$

\begin{theorem}\label{inst-y5}
Let $H$ be a vector space of dimension $n$. Denote by $M_n(Y_5)$ the set of all $\gamma \in A\otimes S^2H^*$
which satisfy the following conditions
\begin{enumerate}
\item the map $\gamma':H\otimes\CU \to H^*\otimes\CU^\perp$ is a fiberwise monomorphism of vector bundles;
\item the rank of the map $\hat\gamma:H\otimes V \to V^*\otimes H^*$ equals $4n + 2$.
\end{enumerate}
Then the moduli space $\CMI_n(Y_5)$ of instantons of charge $n$ on $Y_5$ is the quotient $M_n(Y_5)/\GL(H)$.
In particular, any instanton of charge $n$ is the cohomology bundle of a monad
\begin{equation}\label{my5}
0 \to H\otimes\CU \xrightarrow{\ \gamma'\ } H^*\otimes\CU^\perp \to C\otimes\CO_{Y_5} \to 0,
\end{equation}
where $\gamma \in M_n(Y_5)$ and $C = \Coker\hat\gamma \cong \kk^{n-2}$.
%Moreover, the map $M_n(Y_5) \to \CM_n(Y_5)$ is a principal $\GL_n$ bundle.
\end{theorem}
\begin{proof}
First, let us construct a map $M_n(Y_5) \to \CMI_n(Y_5)$. Take $F = \Coker(\gamma':H\otimes\CU \to H^*\otimes\CU^\perp)$.
Then $F$ satisfies the conditions of Theorem~\ref{ftoe}. Indeed, the only nontrivial thing to check is that
$h^0(F^*) = n-2$. But from the exact sequence
$$
0 \to F^* \xrightarrow{\ \ \ } H\otimes(V/\CU) \xrightarrow{\ \gamma'\ } H^*\otimes\CU^* \to 0
$$
it follows that $H^0(Y_5,F^*)$ is the kernel of the map $H\otimes V \to H^*\otimes V^*$ induced by $\gamma'$.
It is clear that this map coincides with $\hat\gamma$, hence its rank is $4n+2$, so the kernel has dimension $5n - (4n+2) = n-2$.
So, we deduce that $F$ is the acyclic extension of an instanton $E$ of charge $n$ which is the cohomology of the monad
$$
0 \to H\otimes\CU \to H^*\otimes\CU^\perp \to \CO_{Y_5}^{n-2} \to 0.
$$
This construction can be performed in families, so we obtain a morphism $M_n(Y_5) \to \CMI_n(Y_5)$.
This morphism is surjective by Proposition~\ref{ftog}. So, it remains to check that the fibers are
the orbits of $\GL(H)$.

Indeed, assume that the instantons $E_1$ and $E_2$ constructed from $\gamma_1,\gamma_2 \in A\otimes S^2H^*$ are isomorphic.
In other words, the cohomology bundles of the monads
$$
\xymatrix@1{0 \ar[r] & H\otimes\CU \ar[r]^-{\gamma'_1} & H^*\otimes\CU^\perp \ar[r] & \CO_{Y_5}^{n-2} \ar[r] & 0}
\quad\text{and}\quad
\xymatrix@1{0 \ar[r] & H\otimes\CU \ar[r]^-{\gamma'_2} & H^*\otimes\CU^\perp \ar[r] & \CO_{Y_5}^{n-2} \ar[r] & 0}
$$
are isomorphic. Since the monads come from a decomposition with respect to an exceptional collection,
the isomorphism extends to an isomorphism of monads. Thus there are unique isomorphisms $f:H \to H$ and $g:H^* \to H^*$
such that $\gamma'_2\circ f  = g\circ \gamma'_1$. Transposing (and using symmetricity of $\gamma_i$) we obtain
$\gamma'_1\circ g^T = f^T\circ \gamma'_2$. Multiplying with $f^{-T}$ on the left and $g^{-T}$ on the right we obtain
$\gamma'_2\circ g^{-T} = f^{-T}\circ \gamma'_1$. Since $f$ and $g$ are unique it follows that $g = f^{-T}$,
hence $\gamma'_1 = f^T\circ \gamma'_2\circ f$.
\end{proof}

One can rewrite slightly the monad as follows. Note that the morphism $H^*\otimes\CU^\perp \to C\otimes\CO_{Y_5}$
factors as $H\otimes\CU^\perp \to H^*\otimes V^*\otimes\CO_{Y_5} \to C \otimes\CO_{Y_5}$.
Therefore we have the following commutative diagram
$$
\xymatrix{
&& H^*\otimes\CU^\perp \ar[rr] \ar@{=}[d] && H^*\otimes V^*\otimes \CO_{Y_5} \ar[rr] \ar[d] && H^*\otimes\CU^* \\
H\otimes\CU \ar[rr]^-{\gamma'} && H^*\otimes\CU^\perp \ar[rr] && C\otimes \CO_{Y_5}
}
$$
Since the top row is acyclic, it follows that the bottom row is quasi-isomorphic to
\begin{equation}\label{sdm5}
\xymatrix@1{0 \ar[r] & H\otimes\CU \ar[r] & K\otimes\CO_{Y_5} \ar[r] & H^*\otimes\CU^* \ar[r] & 0},
\end{equation}
where $K = \Ker(H^*\otimes V^* \to C) = \Im\hat\gamma$. So, we have proved

\begin{proposition}\label{inst-y5-1}
Any instanton of charge $n$ on $Y_5$ is the cohomology of a self-dual monad~\eqref{sdm5} with $\dim H = n$, $\dim K = 4n+2$.
\end{proposition}

\subsection{Instantonic nets of quadrics }

Any tensor $\gamma \in A\otimes S^2H^*$ can be thought of as a net of quadrics in $\PP(H)$ parameterized
by $\PP(A^*)$. So, given an instanton $E$ on $Y_5$ we can consider the corresponding net of quadrics $\gamma_E$.

The space of nets of quadrics, $A\otimes S^2H^*$, is acted upon by the group $\GL(H)$, so one can
speak about GIT stability and semistability of a net of quadrics. Recall that according to~\cite{W}
a net $\gamma$ is unstable if and only if there is a pair of subspaces $H_1,H_2 \subset H$
such that
\begin{itemize}
\item $\dim H_1 + \dim H_2 > \dim H$, and
\item the map $A^* \xrightarrow{\ \gamma\ } S^2H^* \to H_1^*\otimes H_2^*$ is zero.
\end{itemize}

\begin{proposition}\label{gsta}
For any instanton $E$ on $Y_5$ the corresponding net of quadrics $\gamma_E$ is semistable.
%The net $\gamma_E \in A\otimes S^2H^*$ is generically nondegenerate, i.e.\ the image of the map
%$\gamma_E:A^* \to S^2H^*$ contains nondegenerate quadrics.
\end{proposition}
\begin{proof}
Assume that $\gamma_E$ is unstable. Let $(H_1,H_2)$ be the destabilizing pair of subspaces.
Consider the subspace $H_2^\perp := \Ker(H^* \to H_2^*)$. Note that the condition $\dim H_1 + \dim H_2 > \dim H$
is equivalent to
$$
\dim H_1 > \dim H_2^\perp.
$$
The second condition says that the image of the map
$H_1\otimes A^* \subset H\otimes A^* \xrightarrow{\ \gamma_E\ } H^*$
is contained in $H_2^\perp$.
%
%Assume that the image is contained in the subset of degenerate quadrics.
%Then there is a pair of subspaces $H_1 \subset H_2 \subset H$ such that
%$\dim H_1 + \dim H_2 > \dim H$ and $\Im(\gamma_E) \subset \Ker (S^2H^* \to H_1^*\otimes H_2^*)$
%({\tt THIS SHOULD BE EXPLAINED!!!}). The above condition means that the composition
%$$
%H_1\otimes A^* \to H\otimes A^* \xrightarrow{\ \gamma_E\ } H^* \to H_2^*
%$$
%vanishes, hence the product of the first two maps factors through $H_2^\perp$.
%
Thus we have a commutative diagram
$$
\xymatrix{
H_1\otimes A^* \ar[r] \ar[d] & H_2^\perp \ar[d] \\
H\otimes A^* \ar[r]^{\gamma_E} & H^*
}
$$
Consider the map $\gamma_s:H_1\otimes\CU \to H_2^\perp\otimes\CU^\perp$ induced by the upper line of the above diagram
and the induced diagram
%Denote its kernel by $K$
%It gives as a consequence the following commutative diagram
$$
\xymatrix{
0 \ar[r] & H_1\otimes \CU \ar[r] \ar[d]^{\gamma_s} & H\otimes \CU \ar[r] \ar[d]^{\gamma_E} & (H/H_1)\otimes\CU \ar[r] \ar[d]^{\gamma_q} & 0 \\
0 \ar[r] & H_2^\perp\otimes\CU^\perp \ar[r] & H^*\otimes\CU^\perp \ar[r] & H_2^*\otimes\CU^\perp \ar[r] & 0
}
$$
with exact rows. Since the morphism $\gamma_E$ is injective by Proposition~\ref{ftog} we conclude that $\gamma_s$ is injective as well.
Moreover, we obtain an exact sequence
$$
0 \to \Ker \gamma_q \to \Coker \gamma_s \to \Coker \gamma_E \to \Coker \gamma_q \to 0.
$$
Note that by stability of $\TE \cong \Coker\gamma_E$ the image of the middle arrow should have nonpositive first Chern class,
hence
$$
%c_1(\Coker\gamma_q) \ge 0,
%\qquad
c_1(\Coker\gamma_s) \le c_1(\Ker\gamma_q).
$$
On the other hand since $\Ker\gamma_s=0$ we have
%
%Moreover, we obtain an induced map $\Coker\gamma_0 \to \Coker\gamma_E = \TE$, where $\TE$ is the acyclic extension of the instanton $E$
%(again by Proposition~\ref{ftog}). Now note that $\dim H_1 > \dim H_2^\perp$ implies
$$
c_1(\Coker\gamma_s) = -\dim H_2^\perp + \dim H_1 > 0,
$$
hence $c_1(\Ker\gamma_q) > 0$. But $\Ker\gamma_q$ is a subsheaf in $(H/H_1)\otimes\CU$, and $\CU$ is stable of negative slope.
This contradiction proves the claim.
%
%
%while $c_1(F_E) = 0$. Since the acyclic extension $\TE$ is semistable by Lemma~\ref{te1} we conclude that
%the map $\Coker\gamma_0 \to \TE$ should be zero. But this implies that the composition
%$H_2^\perp\otimes\CU^\perp \to H^*\otimes\CU^\perp \to \TE$ vanishes, hence the first map
%factors through $H\otimes\CU$. But $\Hom(\CU^\perp,\CU) = 0$, hence this is possible only
%when $H_2^\perp = 0$. In this case injectivity of the map $\gamma_0$ implies that $H_1 = 0$
%as well, hence $\dim H_1 + \dim H_2 = \dim H$. This contradiction completes the proof.
%
\end{proof}

\subsection{Jumping lines}

Again consider the net of quadrics $\gamma \in A\otimes S^2H^*$ associated with an instanton $E$.
Assume for a moment that generic quadric in the net is nondegenerate. Then degenerate quadrics form
a curve (of degree $n$) in $\PP(A^*)$ which we denote by $D_\gamma$.
By definition the curve $D_\gamma$ is the support of the cokernel of the morphism
$H\otimes\CO_{\PP(A^*)}(-2) \xrightarrow{\ \gamma\ } H^*\otimes\CO_{\PP(A^*)}(-1)$ induced by $\gamma$.
The cokernel itself is a coherent sheaf (we denote it by $\theta_\gamma$) with the property that
\begin{equation}\label{theta}
\RCHom(\theta_\gamma,\omega_{D_\gamma}) \cong \theta_\gamma.
\end{equation}
In particular, if the net is regular, the curve $D_\gamma$ is smooth and $\theta_\gamma$
is a theta-characteristic, that is a line bundle which is a square root of the canonical class.
Moreover, as the defining exact sequence
\begin{equation}\label{dgamma}
0 \to H\otimes\CO_{\PP(A^*)}(-2) \xrightarrow{\ \gamma\ } H^*\otimes\CO_{\PP(A^*)}(-1) \xrightarrow{\ \ \ } \theta_\gamma \to 0
\end{equation}
shows, this theta-characteristic is {\em non-degenerate}, that is
\begin{equation}\label{tnd}
H^0(D_\gamma,\theta_\gamma) = 0.
\end{equation}
In case of a nonregular net the sheaf $\theta_\gamma$ is neither locally free nor of rank 1 in general.
But still it enjoys the properties~\eqref{theta} and~\eqref{tnd}. We will call such sheaves
{\sf generalized nondegenerate theta-characteristics}.

%In particular, we can consider its degeneration curve $D_\gamma \subset \PP(A^*)$
%(which is a pane curve of degree $n = \dim H$) and the induced theta-characteristic $\theta_\gamma$ on $D_\gamma$.

%The curve of jumping lines of the instanton can be easily reconstructed from the net of quadrics $\gamma$.
%For this consider the morphism $H\otimes\CO_{\PP(A^*)}(-2) \to H^*\otimes\CO_{\PP(A^*)}(-1)$ induced by $\gamma \in A^*\otimes S^2H^*$.
%It is well known that if the net of quadrics is regular than the cokernel is a nondegenerate theta-characteristic
%on a plane curve of degree $n = \dim H$. Denote the curve by $D_\gamma$ and the theta-characteristic by $\theta_\gamma$.

Recall that the Fano scheme of lines on $Y_5$ coincides with $\PP(A)$ which itself is identified with $\PP(A^*)$,
so the curve $D_\gamma$ can be thought of as a curve on the Fano scheme of lines. It turns out that it coincides
with the curve of jumping lines of the instanton $E_\gamma$, and the corresponding sheaf $\CL_E$ is obtained
from the theta-characteristic $\theta_\gamma$ by a twist.

\begin{proposition}\label{dgde}
Let $E$ be an instanton on $Y_5$ and $\gamma_E$ the corresponding net of quadrics. Then one has a distinguished triangle
$$
Rp_*q^*E(-1) \xrightarrow{\qquad} H \otimes \CO_{\PP(A^*)}(-3) \xrightarrow{\ \gamma_E\ } H^* \otimes \CO_{\PP(A^*)}(-2).
$$
In particular, generic line is nonjumping for $E$ if and only if generic quadric in the net $\gamma_E$ is nondegenerate.
Furthermore, if these equivalent conditions hold then $D_E = D_\gamma$ and $\CL_E = \theta_\gamma(-1)$.
\end{proposition}
\begin{proof}
By Lemma~\ref{tejl} we know that $Rp_*q^*E(-1) \cong Rp_{1*}\RCHom(q_1^*\TE,\CJ)$.
On the other hand, one can easily write a relative version of~\eqref{ily5}
$$
0 \to \CU\boxtimes\CO_{\PP(A^*)}(-3) \xrightarrow{\quad} \CU^\perp\boxtimes\CO_{\PP(A^*)}(-2) \xrightarrow{\quad} I_Z \to 0
$$
which gives a distinguished triangle
$$
(V/\CU)(-1) \boxtimes \CO_{\PP(A^*)}(-3) \xrightarrow{\quad} \CU \boxtimes \CO_{\PP(A^*)}(-2) \xrightarrow{\quad} \CJ.
$$
Now we combine this triangle with the exact sequence
$$
H\otimes\CU \xrightarrow{\ \gamma_E\ } H^*\otimes\CU^\perp \xrightarrow{\quad} \TE.
$$
Note that $\Ext^\bullet(\CU,(V/\CU)(-1)) = \Ext^\bullet(\CU^\perp,\CU) = 0$ by Lemma~\eqref{by5},
$\Ext^\bullet(\CU,\CU) = \kk$ since $\CU$ is exceptional and $\Ext^\bullet(\CU^\perp,(V/\CU)(-1)) = \kk[-1]$ by~\eqref{seqy5}.
This gives the desired distinguished triangle
%
%So, we conclude that there is a distinguished triangle
$$
Rp_{1*}\RCHom(q_1^*\TE,\CJ) \to H \otimes \CO_{\PP(A^*)}(-3) \xrightarrow{\ \gamma_E\ } H^* \otimes \CO_{\PP(A^*)}(-2).
$$
The rest of the Proposition easily follows.
\end{proof}

%Note that the above argument also shows that for a given instanton generic line is nonjumping
%if and only if the generic quadric in the corresponding net is nondegenerate.

The above Proposition gives the following reinterpretation of Conjecture~\ref{jl-conj} in terms
of the associated net of quadrics --- if $\gamma$ is an instantonic net of quadrics then generic quadric
in the net is nondegenerate. In fact we believe that this should follow from the semistability of the net.
To be more precise, we have the following

\begin{conjecture}\label{ssnq}
If $\gamma$ is a semistable net of quadrics then generic quadric is nondegenerate.
\end{conjecture}

\begin{remark}
Analogous statement for pencils of quadrics is very easy to prove by analyzing the possible
isomorphism classes of the images of the map $H\otimes\CO_{\PP^1}(-1) \to H^*\otimes\CO_{\PP^1}$
given by the pencil. If the image is $\CO_{\PP^1}^a \oplus \CO_{\PP^1}(-1)^b$ with $a + b < \dim H$
then taking $H_1 = \Ker(H\otimes\CO_{\PP^1}(-1) \to \CO_{\PP^1}(-1)^b)$ and $H_2 = \Coker(\CO_{\PP^1}^a \to H^*\otimes\CO_{\PP^1})^*$
we get a destabilizing pair of subspaces.

On the other hand, for higher dimensional linear spaces of quadrics the analogous statement is wrong.
For example, the 5-dimensional space of Pl\"ucker equations of $\Gr(2,5)$ consists of degenerate
quadrics, but is stable.
\end{remark}

We can also use Proposition~\ref{dgde} to deduce Conjecture~\ref{jlines}.

\begin{corollary}
For Fano threefold of degree $5$ Conjecture~{\rm\ref{jlines}} is true.
\end{corollary}
\begin{proof}
By Proposition~\ref{dgde} the (generalized) theta-characteristic of the net can be reconstructed
from the sheaf $\CL_E$ on $D_E$, so it suffices to recall that the net can be reconstructed from
the associated theta-characteristic $\theta$. Indeed, if we consider $\theta$ as a sheaf on the projective
plane, then the complex~\eqref{dgamma} is nothing but the decomposition of $\theta$ with respect
to the standard exceptional collection $(\CO(-2),\CO(-1),\CO)$ (by nondegeneracy property $\theta$
is orthogonal to $\CO$, so it doesn't appear in the decomposition). But the morphism $H\otimes\CO(-2) \to H^*\otimes\CO(-1)$
gives back the net. Finally, the net allows to reconstruct the instanton by Theorem~\ref{inst-y5} (or Proposition~\ref{inst-y5-1}).
\end{proof}

\section{Instantons on Fano threefolds of degree $4$}\label{sy4}

In this section we concentrate on Fano threefolds of degree 4.

\subsection{Derived category}

A Fano threefold of degree 4 and index 2 is an intersection of 2 quadrics in~$\PP^5$. Denote by $V$ a vector space of dimension 6
and by $A$ a vector space of dimension 2. Then a pair of quadrics gives a map $A \to S^2V^*$, so we have a family of quadrics
in $\PP(V)$ parameterized by $\PP(A)$. There are 6 degenerate quadrics in this family, giving 6 special points $a_1,\dots,a_6 \in \PP(A)$.
Let $C$ be the double covering of $\PP(A)$ ramified in $\{a_1,\dots,a_6\}$. Then $C$ is a curve of genus 2. Denote by $\pi:C \to \PP(A)$
the double covering and by $\tau:C \to C$ its hyperelliptic involution. We will need the following description of the category $\CB_{Y_4}$

\begin{theorem}[\cite{BO1,K08a}]
There is an equivalence $\CB_{Y_4} \cong \D^b(C)$ given by the Fourier--Mukai functor associated
with the family of spinor bundles on the quadrics in the family $\PP(A)$.
\end{theorem}

Let us explain the statement. On each smooth quadric in the family $\PP(A)$ there are two spinor bundles. Restricting them to $Y_4$
we obtain a pair of bundles on $Y_4$ which can be thought of as being associated with two points of $C$ over the point of $\PP(A)$
corresponding to the quadric. Similarly, each singular quadric in $\PP(A)$ is a cone over a 3-dimensional quadric and $Y_4$ does not
pass through its vertex. Hence the projection from the vertex gives a map from $Y_4$ onto a 3-dimensional quadric and we can pullback
its (unique!) spinor bundle to $Y_4$. This gives a bundle associated with the branching point of $C \to \PP(A)$. One can show that
all those spinor bundles form a vector bundle $\CS$ of rank 2 on $C \times Y_4$ and the Fourier--Mukai functor
$\Phi_\CS:\D^b(C) \to \D^b(Y_4)$ is an equivalence onto $\CB_{Y_4}$. Note that this defines $\CS$ only up to a twist
by the pullback of a line bundle on $C$.

Another approach to the relation of $C$ and $Y_4$ and the description of the universal spinor bundle $\CS$ on $C \times Y_4$
is due to Mukai. He showed that $Y_4$ is the moduli space of stable rank 2 vector bundles on $C$ with fixed determinant $\xi$
of odd degree and that $\CS$ is the universal family for this moduli problem. For our convenience we assume that
$$
\deg \xi = 1
$$
(note that a twist by a line bundle of degree $k$ changes the degree of the determinant of a rank 2 bundle by $2k$,
so the moduli spaces for all odd degrees are isomorphic and the corresponding universal spinor bundles $\CS$ differ
by the corresponding twists). This fixes the bundle $\CS$ unambiguously. In particular, we have
$$
\det \CS = \xi \boxtimes \CO_{Y_4}(-1).
$$
In fact one can compute also
$$
c_2(\CS) = \eta + 2 L_Y,
\qquad \eta \in H^1(C)\otimes H^3(Y_4) \subset H^4(C\times Y_4),
\qquad \eta^2 = 4p_C p_Y,
$$
where $H_Y$, $L_Y$, and $p_Y$ stand for the classes of a hyperplane section,
of a line and of a point on $Y_4$, while $p_C$ stands for the class of a point on $C$.
This allows to write down the Grothendieck--Riemann--Roch for the functor $\Phi = \Phi_\CS$.

\begin{lemma}\label{GRR}
For any $F \in\D^b(C)$ we have
% If $F \in \D^b(C)$ is such that $r(F) = r$, $\deg F = d$ then
$$
% \ch(\Phi(F)) = (2d - r) - d H_Y + r L_Y + \frac{d}3 P_Y.
\ch(\Phi(F)) = (2\deg(F) - r(F)) - \deg(F) H_Y + r(F) L_Y + \frac{\deg(F)}3 P_Y.
$$
\end{lemma}
\begin{proof}
One has
$$
\ch(\CS) = 2 + (p_C - H_Y) - p_C H_Y - \eta + p_CL_Y + \frac13 p_Y + \frac13 p_Cp_Y.
$$
Since the relative tangent bundle of $C\times Y \to Y$ is just the pullback of $\omega_C^{-1}$,
its Todd genus equals $1 - p_C$, so
$$
\ch(\CS)\td(T_C) = 2 - p_C - H_Y - \eta + p_CL_Y + \frac13 p_Y.
$$
Multiplying this by $\ch(F) = r(F) + \deg(F) p_C$ and taking pushforward to $Y_4$ (i.e. taking the coefficient at $p_C$)
one obtains the result.
\end{proof}

\subsection{Lines}

The description of the Fano scheme of lines on $Y_4$ is well known.
However, for our purposes we will need a description closely related
to our Fourier--Mukai functor. We start with the following

\begin{lemma}\label{syl}
Let $\CL$ be a line bundle of degree $0$ on $C$ and $\CS_y$ a stable rank $2$ vector bundle on $C$
with $\det\CS_y = \xi$ corresponding to a point $y \in Y_4$. If $H^0(C,\CL\otimes\CS_y) \ne 0$
then $\CS_y$ is a nontrivial extension
\begin{equation}\label{cly}
0 \to \CL^{-1} \to \CS_y \to \CL\otimes\xi \to 0.
\end{equation}
Vice versa, $\Ext^1(\CL\otimes\xi,\CL^{-1}) = \kk^2$ and each nontrivial extension
of $\CL\otimes\xi$ with $\CL^{-1}$ is a stable rank $2$ bundle on $C$ with determinant $\xi$.
\end{lemma}
\begin{proof}
Assume that $H^0(C,\CL\otimes\CS_y) \ne 0$. Then we have a map $\CL^{-1} \to \CS_y$.
If this map is not injective at a point $x \in C$ then the map factors through $\CL^{-1}(x)$
which is impossible by stability of $\CS_y$ (since $\deg\CL^{-1}(x) = 1$). So, the map
$\CL^{-1} \to \CS_y$ is an embedding of vector bundles. Hence the quotient is a line bundle
which has to be isomorphic to $\det\CS_y \otimes \CL \cong \CL\otimes\xi$. The extension
is nontrivial since $\CS_y$ is simple.

Vice versa, note that $\Ext^\bullet(\CL\otimes\xi,\CL^{-1}) = H^\bullet(C,\CL^{-2}\otimes\xi^{-1})$.
Since $\deg(\CL^{-2}\otimes\xi^{-1}) = -1$, there are no global sections and by Riemann--Roch
the first cohomology has dimension $2$. Now take any nontrivial extension
$$
0 \to \CL^{-1} \to \CE \to \CL\otimes\xi \to 0.
$$
Evidently $\det\CE = \xi$, so let us check that $\CE$ is stable. If not then there should be a line bundle
$\CL'$ of degree~$1$ such that $\Hom(\CL',\CE) \ne 0$. Applying $\Hom(\CL',-)$ to the above exact sequence
we obtain
$$
0 \to \Hom(\CL',\CL^{-1}) \to \Hom(\CL',\CE) \to \Hom(\CL',\CL\otimes\xi) \to \Ext^1(\CL',\CL^{-1}) \to \dots
$$
Since $\deg \CL' = 1$ and $\deg \CL^{-1} = 0$ the first term is zero. Further,
since $\deg \CL\otimes\xi = 1$ the third term is nontrivial only if $\CL' = \CL\otimes\xi$.
In the latter case the map from the third term to the fourth term is the map
$\kk \to \Ext^1(\CL\otimes\xi,\CL^{-1})$ given by the class of the extension,
so if the extension is nontrivial the map is injective and we have $\Hom(\CL',\CE) = 0$ in any case.
\end{proof}

Also we will need the following simple observation.

\begin{lemma}\label{ltl}
For any line bundle $\CL$ on a curve of genus $2$ one has $\CL\otimes\tau^*\CL \cong \omega_C^{\deg\CL}$.
In particular, if $\deg\CL = 0$ then $\CL^*\cong\tau(\CL)$.
\end{lemma}
\begin{proof}
First take $\CL \cong \CO_C(x)$ for some point $x \in C$. Then $\tau^*\CL \cong \CO_C(\tau(x))$ and
$\CL\otimes\tau^*\CL \cong \CO_C(x + \tau(x))$. But $x + \tau(x)$ is the preimage of a point under
the projection $C \to \PP^1$, hence the corresponding line bundle is the canonical class.
This proves the formula for $\CL = \CO_C(x)$. After that the general case follows since
any line bundle is a (multiplicative) linear combination of line bundles $\CO_C(x)$,
and both sides of the formula are (multiplicatively) linear in $\CL$.
\end{proof}

The set of points $y \in Y_4$ for which the bundle $\CS_y$ fits into exact triple~\eqref{cly} is a curve
isomorphic to $\PP(\Ext^1(\CL\otimes\xi,\CL^{-1})) = \PP^1$. We denote this curve by $L_\CL \subset Y_4$.
Below we will show that it is a line on $Y_4$.

Recall that with each line $L \subset Y_4$ we associate two objects,
the ideal sheaf $I_L \in \CB_{Y_4}$ and the object $J_L = \RCHom(I_L,\CO_{Y_4}(-1))[1] \in \CB_{Y_4}$ as well.
%Let us identify the object in $\D^b(C)$ which is taken to $I_L$ by $\Phi$.

\begin{lemma}\label{phid01}
The are isomorphism $\phi_0:F(Y_4) \xrightarrow{\ \sim\ } \Pic^0(C)$ and $\phi_1:F(Y_4) \xrightarrow{\ \sim\ } \Pic^1(C)$ given by
\begin{equation*}
\phi_0(L) = \Phi^{-1}(I_L[-1]),
\qquad
\phi_0(L) = \Phi^{-1}(J_L),
\end{equation*}
% 
% 
% $L \mapsto \Phi^{-1}(I_L[-1])$ and $L \mapsto \Phi^{-1}(J_L)$ give isomorphisms
% of the Fano scheme $F(Y_4)$ of lines on $Y_4$ onto $\Pic^0 C$ and $\Pic^1 C$ respectively.
Moreover, the diagram
$$
\xymatrix@C=2cm{
& F(Y_4) \ar[dl]_{\phi_0} \ar[dr]^{\phi_1} \\
\Pic^0C \ar[rr]^-{\CL \ \mapsto\  \CL^*\otimes\omega_C\otimes\xi^{-1}} && \Pic^1C
}
$$
is commutative.
\end{lemma}
\begin{proof}
Let $\CF = \Phi^{-1}(I_L[-1])$, so that $\Phi(\CF) = I_L[-1]$.
Then for each point $x \in C$ we have
$$
\Ext^\bullet(\CF,\CO_x) \cong \Ext^\bullet(\Phi(\CF),\Phi(\CO_x)) \cong \Ext^\bullet(I_L[-1],\CS_x) \cong \Ext^\bullet(\CO_L[-2],\CS_x)
$$
(the last isomorphism follows from exact sequence $0 \to I_L \to \CO_{Y_4} \to \CO_L \to 0$ since we have $\CS_x \in \CB_{Y_4}$).
Note that $\CS_x$ is a vector bundle of rank 2 and degree $-1$, and its dual is globally generated.
Hence $(\CS_x)_{|L} = \CO_L \oplus \CO_L(-1)$, therefore $\Ext^\bullet(\CO_L,\CS_x) = \kk[-2]$.
We conclude that $\Ext^\bullet(\CF,\CO_x) \cong \kk$ for all $x \in C$, hence $\CF \cong \CL$
where $\CL$ is a line bundle. Since $c_1(I_L[-1]) = 0$ we deduce from~\eqref{GRR} that $\deg\CL = 0$,
that is $\CL \in \Pic^0C$.

Vice versa, let $\CL \in \Pic^0C$. Since $\Phi(\CL)$ is the derived pushforward of a vector bundle $p_1^*\CL\otimes\CS$ on $C\times Y_4$
along the projection $C\times Y_4 \to Y_4$, its cohomology sheaves a priori sit in degrees $0$ and $1$.
We denote those by $\CH^0$ and $\CH^1$ respectively.
Note that we have
$$
\CH^\bullet(j_y^*\Phi(\CL)) \cong H^\bullet(C,\CL\otimes\CS_y),
$$
where $j_y:\Spec\kk \to Y_4$ is the embedding of the point $y$.
By Lemma~\ref{syl} we have
$$
%\CH^0(j_y^*\Phi(\CL)) \cong
H^0(C,\CL\otimes\CS_y) = \begin{cases} \kk, & \text{if $y \in L_\CL$} \\ 0, & \text{if $y \not\in L_\CL$} \end{cases}
\qquad
%\CH^1(j_y^*\Phi(\CL)) \cong
H^1(C,\CL\otimes\CS_y) = \begin{cases} \kk^2, & \text{if $y \in L_\CL$} \\ \kk, & \text{if $y \not\in L_\CL$} \end{cases}
$$

On the other hand, we have a spectral sequence
$$
L_tj_y^*\CH^s \Longrightarrow \CH^{s-t}(j_y^*\Phi(\CL))
$$
which can be rewritten as a long exact sequence
$$
0 \to L_2j_y^*\CH^1 \to L_0j_y^*\CH^0 \to H^0(C,\CL\otimes\CS_y) \to L_1j_y^*\CH^1 \to 0,
$$
and isomorphisms
$$
L_0j_y^*\CH^1 = H^1(C,\CL\otimes\CS_y),
\qquad
L_tj_y^*\CH_0 = L_{t+2}j_y^*\CH^1\quad\text{for $t \ge 1$}.
$$
It follows that for generic $y \in Y_4$ we have $L_\bullet j_y^*\CH^0 = 0$, hence the support of $\CH^0$
is a proper subvariety of $Y_4$. On the other hand, $\CH^0 = R^0{p_2}_*(p_1^*\CL\otimes\CS)$ is torsion free,
hence $\CH^0 = 0$. Thus the above formulas say that
$$
L_0j_y^*\CH^1 = H^1(C,\CL\otimes\CS_y),
\qquad
L_1j_y^*\CH^1 = H^0(C,\CL\otimes\CS_y),
\qquad
L_{\ge 2}j_y^*\CH^1 = 0.
$$
In other words, the sheaf $\CH^1$ is locally free of rank $1$ on $Y_4 \setminus L_\CL$ and has a singularity
along a curve $L_\CL$. Note that it follows that $\CH^1$ is torsion free. Indeed, if $\CH^1$ would have a torsion,
its support would lie in $\CL_L$, hence would have codimension at least 2, hence $L_2ij_y^*\CH^1$ would be nonzero
for any point $y$ in the support of the torsion subsheaf, while we know that it is zero.

Thus we know that $\CH^1$ is a torsion free sheaf of rank $1$. Moreover, by Lemma~\ref{GRR}
its Chern character equals
$$
\ch(\CH_1) = -\ch(\Phi(\CL)) = 1 - L_{Y_4}.
$$
In particular, $c_1(\CH_1) = 0$, hence $\CH_1$ is the sheaf of ideals of a subscheme $Z$,
$\CH^1 \cong I_Z$, where $Z$ is a subscheme set-theoretically supported on $L_\CL$ and such that
$$
\ch(\CO_Z) = L_{Y_4}.
$$
It follows that $Z$ is a line, but possibly with a non-reduced structure at some points.
However, if $Z$ would have a non-reduced structure at a point $y$, then $\CO_Z$ would have
a subsheaf supported at this point and then $L_3j_y^*\CO_Z \ne 0$, hence $L_2j_y^*I_Z \ne 0$
which is a contradiction. Thus $Z$ is a line, hence $L_\CL$ is a line and $\Phi(\CL) = I_{L_\CL}[-1]$.

This proves that $\Phi$ induces an isomorphism of $\Pic^0\CL$ with $F(Y_4)$ considered as
the moduli space of sheaves of ideals of lines, hence $\phi_0$ is an isomorphism of $F(Y_4)$ onto $\Pic^0C$. 
To relate $F(Y_4)$ with $\Pic^1(C)$ we recall that $J_L = \RCHom(I_L[-1],\CO_{Y_4}(-1))$, hence
\begin{multline*}
J_L =
\RCHom(\Phi(\CL),\CO_{Y_4}(-1)) =
\RCHom(Rp_{Y*}(\CS\otimes p_C^*\CL),\CO_{Y_4}(-1)) \cong \\ \cong
Rp_{Y*}\RCHom(\CS\otimes p_C^*\CL,p_Y^!\CO_{Y_4}(-1)) \cong
Rp_{Y*}\RCHom(\CS\otimes p_C^*\CL,p_C^*\omega_C\otimes p_Y^*\CO_{Y_4}(-1)[1]) \cong \\ \cong
Rp_{Y*}(\CS^*\otimes p_C^*(\CL^*\otimes\omega_C)\otimes p_Y^*\CO_{Y_4}(-1)[1]),
\end{multline*}
where $p_Y$ and $p_C$ are the projections of $C\times Y_4$ onto the factors $Y_4$ and $C$ respectively.
Note also that $\CS^*\otimes p_C^*\xi\otimes p_Y^*\CO_{Y_4}(-1) \cong \CS$ since $\CS$ is a vector bundle
of rank 2 with determinant equal to $\xi \boxtimes \CO_{Y_4}(-1)$. Hence we conclude that
$$
J_L \cong Rp_{Y*}(\CS\otimes p_C^*(\CL^*\otimes\omega_C\otimes\xi^{-1})[1]) = \Phi(\CL^*\otimes\omega_C\otimes\xi^{-1})[1]
$$
which gives the commutativity of the diagram. Since both the left and the bottom arrows in the diagram
are isomorphisms, we conclude that the right arrow is an isomorphism as well.
%
%
%%the fiber of $\Phi(\CL)$ at a point $y \in Y_4$ is given by $H^\bullet(C,\CL\otimes\CS_y)$.
%It follows immediately that $\CH^{>1}(\Phi(\CL)) = 0$
%
%Since $\CL\otimes\CS_y$ is a stable vector bundle of rank 2 and degree 1
%we have $H^t(C$
%
%
%
%Since $\Phi$ is an equivalence of $\D^b(C)$ onto $\CB_{Y_4}$ to any object of the latter category
%a unique object of the former is mapped. In particular, for each line $L$ there is a unique $\CF \in \D^b(C)$
%such that
%
%
%
\end{proof}

\begin{lemma}
The image of the divisor $D_L$ in $\Pic^1C$ under the map $\CL \mapsto \CL^*\otimes\omega_C\otimes\xi^{-1}$
is a translate of the theta-divisor by $\CL$.
\end{lemma}
\begin{proof}
Recall that for any lines $L,L'$ on $Y_4$ we can write $I_L = \Phi(\CL)[1]$, $J_{L'} = \Phi(\CL')[1]$, where $\CL \in \Pic^0C$, $\CL' \in \Pic^1C$.
So,
$$
\Hom(I_L,J_{L'}) = \Hom(\Phi(\CL),\Phi(\CL')) = \Hom(\CL,\CL') = H^0(\CL^{-1}\otimes\CL').
$$
Since $\CL^{-1}\otimes\CL'$ is a line bundle of degree $1$, it has a global section if and only if it is isomorphic
to the line bundle $\CO_C(x)$ for some point $x \in C$, that is if $\CL' \cong \CL(x)$. Thus by Lemma~\ref{llpint}
the curve $D_L \subset \Pic^1C$ is the theta-divisor translated by $\CL$.
\end{proof}

\subsection{The action of the antiautoequivalence}

We also can identify the action of the antiautoequivalence $\SD$ on $\D^b(C)$.

\begin{proposition}\label{dy4}
We have $\SD(\CF) \cong \tau^*\CF^*[2]$.
\end{proposition}
\begin{proof}
Since $C$ is a variety of general type we know by~\cite{BO2} that any antiautoequivalence of $\D^b(C)$
is a composition of the usual dualization with a shift, a twist, and an automorphism.
First, let us check how $\SD$ acts on the structure sheaves of points, that is, in terms of $\CB_Y$, on spinor bundles $\CS_x$.
First, note that $H^\bullet(Y,\CS_x^*) = \kk^4$, the induced map $\CO_Y^{\oplus 4} \to \CS_x^*$ is surjective and its kernel
is $\CS_{\tau(x)}$ (this can be checked on the corresponding quadric). Thus $\SD(\CS_x) \cong \CS_{\tau(x)}[1]$. In other words, $\SD(\CO_x) \cong \CO_{\tau(x)}[1]$.
Since $\RCHom(\CO_x,\CO_C) \cong \CO_x[-1]$, we see that the shift part is $[2]$ and the automorphism part is given by $\tau$.
To identify the twist part we apply $\SD$ to a line bundle $\CL$ of degree zero. Since $\Phi(\CL) \cong I_L[-1]$
for some line $L$ on $Y$ and since $\SD(I_L) \cong I_L$ by Proposition~\ref{dil}, we conclude that
$$
\SD(\Phi(\CL)) \cong I_L[1] \cong \Phi(\CL[2]).
$$
Since $\tau^*\CL \cong \CL^*$ by Lemma~\ref{ltl}, the claim follows.
\end{proof}

%Note also that if $\CF$ is a line bundle of degree zero then $\SD(\CF) \cong \CF$ by Lemma~\ref{ltl}.

\subsection{Description of instantons}

Now to get a description of the moduli space of instantons we will need to know $\Phi^!(\CO_Y)$.
It turns out that (up to a shift) it is a very interesting vector bundle on $C$, so called
the {\sf second Raynaud bundle}~\cite{R}. By definition this is the (shift of the)
Fourier--Mukai transform of the bundle $\CO_{\Pic C}(-2\Theta)$ with the kernel
given by the Poincare bundle. Note that the theta divisor on $\Pic C$ is defined only
up to a translation, accordingly the second Raynaud bundle is defined up to a twist by a line
bundle of degree $0$ (so more precisely it would be to speak about the Raynaud class of bundles).
We will need the following important property of the Raynaud class of bundles.

\begin{lemma}[\cite{P}]\label{rayprop}
Let $\CR$ be a semistable vector bundle of rank $4$ and of degree $4$ on a curve $C$ of genus~$2$.
If for any line bundle $\CL$ of degree $0$ on $C$ we have $\Hom(\CL,\CR) \ne 0$ then $\CR$ is a second Raynaud bundle.
\end{lemma}

This property can be used to identify the object $\Phi^!(\CO_{Y_4})$.

%Moreover, for any line bundle $\CL \in \Pic^0C$ one has $\Hom(\CL,\CW) = \Ext^1(\CL,\CW) = \kk$.

\begin{lemma}\label{cw}
We have $\Phi^!(\CO_Y) \cong \CR[1]$, where $\CR$ is a second Raynaud bundle on $C$.
\end{lemma}
\begin{proof}
We have $\Ext^\bullet(\CO_x,\Phi^!(\CO_Y)) = \Ext^\bullet(\Phi(\CO_x),\CO_Y) = \Ext^\bullet(\CS_x,\CO_Y) = H^\bullet(Y,\CS_x^*) \cong \kk^4$.
It follows that $\Phi^!(\CO_Y) \cong \CR[1]$, where $\CR$ is a vector bundle of rank 4. Further we have
$$
\Ext^\bullet(\CL,\CR) \cong
\Ext^\bullet(\CL,\Phi^!(\CO_Y[-1])) \cong
\Ext^\bullet(\Phi(\CL),\CO_Y[-1]) \cong
\Ext^\bullet(I_L[-1],\CO_Y[-1]) = \kk \oplus \kk[-1].
$$
It follows from Riemann--Roch that the degree of $\CR$ is $4$. Also it follows that the main property
of Raynaud bundles is true for the bundle $\CR$. So it only remains to check that $\CR$ is semistable.

First consider $\Phi(\CR) = \Phi(\Phi^!(\CO_{Y_4}))[-1]$. Note that by definition of the mutation functor we have a distinguished triangle
$$
\Phi(\Phi^!(\CO_{Y_4})) \to \CO_{Y_4} \to \LL_{\CB_{Y_4}}(\CO_{Y_4}).
$$ 
On the other hand, since we have a semiorthogonal decomposition $\D^b(Y_4) = \langle \CB_{Y_4},\CO_{Y_4},\CO_{Y_4}(1) \rangle$
we know that $\LL_{\CB_{Y_4}}(\CO_{Y_4}) \cong \SS(\RR_{\CO_{Y_4}(1)}(\CO_{Y_4}))$, where $\SS$ is the Serre functor.
Since $\Ext^\bullet(\CO_{Y_4},\CO_{Y_4}(1)) = V^*$ we deduce that $\RR_{\CO_{Y_4}(1)}(\CO_{Y_4}) \cong T_{\PP(V)|Y_4}[-1]$,
the shift of the tangent bundle to $\PP(V)$ restricted to $Y_4$. Hence $\LL_{\CB_{Y_4}}(\CO_{Y_4}) \cong T_{\PP(V)|Y_4}(-2)[2]$. 
Thus the above triangle shows that $\Phi(\Phi^!(\CO_{Y_4}))$ has two cohomology sheaves, $\CO_{Y_4}$ in degree $0$ and 
$T_{\PP(V)|Y_4}(-2)$ in degree $-1$.

Assume that $0 \to F \to \CR \to G \to 0$ is a destabilizing exact sequence of vector bundles with $F$ stable.
Applying the functor $\Phi$ we get a distinguished triangle
$$
\Phi(F) \to \Phi(\CR) \to \Phi(G)
$$
which gives a long exact sequence of cohomology sheaves
\begin{equation}\label{lesr}
0 \to  \CH^0(\Phi(F)) \to T_{\PP(V)|Y_4}(-2) \to \CH^0(\Phi(G)) \to \CH^1(\Phi(F)) \to \CO_{Y_4} \to \CH^1(\Phi(G)) \to 0
\end{equation}
(note that since $\dim C = 1$ the functor $\Phi$ applied to a sheaf can have cohomology sheaves only in degrees $0$ and $1$).
Now since $r(\CR) = 4$ and $\deg(\CR) = 4$ we have
\begin{itemize}
\item either $r(F) = 1$ and $\deg(F) \ge 2$, 
\item or $r(F) = 2$ and $\deg(F) \ge 3$,
\item or $r(F) = 3$ and $\deg(F) \ge 4$.
\end{itemize}
Consider the first two cases. Note that the slope of $F$ is greater or equal than $3/2$ in these cases.
Note also that for any $y \in Y_4$ we have by Serre duality
$$
H^1(C,\CS_y\otimes F) \cong \Hom(F,\CS_y^*\otimes\omega_C)^*.
$$
The second bundle here has slope $2 - 1/2 = 3/2$ and $F$ in the first two cases has slope which is not smaller. Hence 
by stability of $F$ and $\CS_y$ the above space is zero unless $F \cong \CS_y^*\otimes\omega_C$. Since the above is possible
only for one $y$, we conclude that $\CH^1(\Phi(F))$ is either $0$, or is the structure sheaf of a point. In any case 
its rank and $c_1$ is zero. Thus the rank and $c_1$ of the sheaf $\CH^0(\Phi(F))$ coincide with those of $\Phi(F)$
and so by the Grothendieck--Riemann--Roch formula (Lemma~\ref{GRR}) we have
\begin{equation*}
\mu(\CH^0(\Phi(F))) = -\frac{\deg(F)}{2\deg(F) - r(F)}.
\end{equation*}
Under our assumptions on $F$ this is greater than $-4/5$, the slope of $T_{\PP(V)|Y_4}(-2)$.
This contradicts the stability of the latter bundle (which can be easily shown by using Hoppe's criterion, see Lemma~\ref{hoppe})
excluding the first two cases.

In the last case we have $r(G) = 1$, $\deg(G) \le 0$. Such $G$ can be embedded into appropriate line bundle $\CL$ of degree $0$,
hence $\CH^0(\Phi(G)) \subset \CH^0(\Phi(\CL))$ which was shown to be zero (see the proof of Lemma~\ref{phid01}).
Thus by Lemma~\ref{GRR} we have
\begin{equation*}
r(\CH^1(\Phi(G))) = - r(\Phi(G)) = 1 - 2\deg(G).
\end{equation*}
Since $\deg(G) \le 0$ this is greater or equal than $1$. On the other hand, it follows from~\eqref{lesr} 
that $\CH^1(\Phi(G))$ is a quotient of $\CO_{Y_4}$. This is possible only if $\deg(G) = 0$, so $G = \CL \in \Pic^0(C)$. 
Then as we know $\Phi(\CL) = I_L[-1]$ with $L$ a line. Since $I_L$ is not a quotient of $\CO_{Y_4}$ we get a final contradiction.
\end{proof}

Now we are ready to give a description of instantons on $Y_4$.

\begin{theorem}\label{miy4}
Let $\CR$ be a second Raynaud bundle.
The moduli space of instantons $\CMI_n(Y_4)$ is isomorphic to the moduli space of simple vector bundles
$\CF$ on $C$ of rank $n$ and degree $0$ such that
\begin{eqnarray}
&& \CF^* \cong \tau^*\CF,\label{cfsd}\\
% \qquad\text{and}\qquad
&& H^0(C,\CF\otimes\CS_y) = 0\quad\text{for all $y \in Y_4$},\label{cfso}\\
&& \dim\Hom(\CF,\CR) = \dim\Ext^1(\CF,\CR) = n-2.\label{cfr}
\end{eqnarray}
\end{theorem}
\begin{proof}
For each instanton $E$ consider its acyclic extension $\TE$.
Then as we know $\TE = \Phi(\CF)[-1]$ for some $\CF \in \D^b(C)$.
We are going to show that $\CF$ is a vector bundle.
Indeed, since $\Phi:\D^b(C) \to \CB_{Y_4}$ is an equivalence
we have $\CF = \Phi^*(\TE[-1])$. Since $\Phi^*(\CO_{Y_4}) = 0$ we have $\Phi^*(\TE) = \Phi^*(E)$,
so finally
$$
\CF = \Phi^*(E)[-1].
$$
Further, it is easy to check that the functor $\Phi^*$ is also a Fourier--Mukai transform with the kernel
$\CS^*\otimes q^*\CO_{Y_4}(-2)[3]$. Thus the fiber of the object $\CF$ at a point $x \in C$ is given by
$$
\CF_x = H^{\bullet+2}(Y_4,\CS_x^*\otimes E(-2)),
$$
so our goal is to show that only $H^2$ is nontrivial. First, we note that
$$
H^0(Y_4,\CS_x^*\otimes E(-2)) = \Hom(\CS_x,E(-2)) = 0
$$
by stability of $\CS_x$ and $E$. Similarly, using Serre duality we deduce that
$$
H^3(Y_4,\CS_x^*\otimes E(-2)) = H^0(Y_4,\CS_x\otimes E)^* = \Hom(E,\CS_x)^* = 0
$$
again by stability of $E$ and $\CS_x$. Finally, we note that for any $x \in C$ one has a short exact sequence
$$
0 \to \CS_x^* \to \CO_{Y_4}(1)^4 \to \CS_x^*(1) \to 0
$$
(this is the restriction of the standard exact sequence of spinor bundles from a 4-dimensional quadric).
Since $H^\bullet(Y_4,E(-1)) = 0$ we conclude that
$$
H^1(Y_4,\CS_x^*\otimes E(-2)) = H^0(Y_4,\CS_x^*(1)\otimes E(-2)) = \Hom(\CS_x(1),E) = 0
$$
again by stability of $E$ and $\CS_x$. Thus indeed we have only $H^2$, so $\CF$ is a vector bundle.

Since $\Phi(\CF) \cong \TE[1]$ using Lemma~\ref{GRR} we see that $r(\CF) = n$ and $\deg(F) = 0$.
Moreover, since $\Phi$ is fully faithful and $\TE$ is simple by Lemma~\ref{te1}, we conclude that $\CF$ is simple.

Let us check that $\CF$ enjoys~\eqref{cfsd}, \eqref{cfso}, and~\eqref{cfr}.
The first follows immediately from $\SD(\TE) \cong \TE$ and Lemma~\ref{dy4}.
The second follows from the fact that $\Phi(\CF)$ is a vector bundle shifted by $-1$.
And for the third one can use that by Lemma~\ref{cw}
\begin{multline*}
\Ext^\bullet(\CF,\CR) =
\Ext^\bullet(\CF,\Phi^!(\CO_Y)[-1]) \cong
\Ext^\bullet(\Phi(\CF),\CO_Y[-1]) = \\ =
\Ext^\bullet(\TE[-1],\CO_Y[-1]) =
\Ext^\bullet(\TE,\CO_Y) =
H^\bullet(Y,\TE^*),
\end{multline*}
so~\eqref{cfr} follows from Lemma~\ref{te1}.

Now let us check the inverse statement. If $\CF$ is a vector bundle on $C$ such that~\eqref{cfso} holds
then $\CH^0(\Phi(\CF)) = 0$ and $F := \CH^1(\Phi(\CF))$ is a vector bundle, so one can write $\Phi(\CF) \cong F[-1]$.
By Lemma~\ref{GRR} we deduce that $r(F) = n$ and $c_1(F) = 0$. Since the image of the functor $\Phi$ is $\CB_{Y_4}$
we conclude that $H^\bullet(Y_4,F) = H^\bullet(Y_4,F(-1)) = 0$. Moreover, $\SD(F) \cong F$ by~\eqref{cfsd} and Lemma~\ref{dy4},
and since 
\begin{equation*}
H^i(Y_4,F^*) = \Ext^i(F,\CO_Y) = \Ext^i(\Phi(\CF)[1],\CO_Y) \cong \Ext^i(\CF,\Phi^!(\CO_Y[-1])) \cong \Ext^i(\CF,\CR)
\end{equation*}
we see that~\eqref{cfr} implies $h^0(F^*) = h^1(F^*) = n-2$. Thus Theorem~\ref{ftoe} applies and we conclude
that $F$ is the acyclic extension of appropriate instanton of charge $n$ on $Y_4$.
%
%Note that for each point $y \in Y$ one has
%$$
%H^0(C,\CS_y\otimes\CF_2) = \Hom(\CS_y^*,\CF_2) = 0
%$$
%by stability of $\CS_y^*$ and $\CF_2$
%
%$\CF \subset \CG$ be a destabilizing surjection
%with $\deg\CG < 0$. Moreover, we may assume that $\CG$ itself is stable. Note that $H^0(C,\CG\otimes\CS_y) = \Hom(\CS_y^*,\CG) = 0$
%by stability of $\CG$, hence $\Phi(\CG) = E_\CG[-1]$ with $E_\CG$ being a vector bundle.
%
%
\end{proof}

\subsection{Jumping lines}

The curve $D_E$ of jumping lines of an instanton $E$ together with its natural coherent sheaf $\CL_E$
can be described in terms of the associated vector bundle $\CF_E$ on $C$. Recall that in Lemma~\ref{phid01} 
we have constructed an isomorphism $\phi_1$ of $F(Y_4)$ and $\Pic^1(C)$.

\begin{proposition}\label{ley4}
Let $\CF_E$ be the stable vector bundle on $C$ corresponding to an instanton $E$.
Then isomorphism $\phi_1$ identifies the set of jumping lines $D_E$ of $E$ with the set of $\CL \in \Pic^1C$
such that $\Ext^\bullet(\CF,\CL) \ne 0$. Moreover, let $\CP$ be the Poincare line bundle
on $C \times \Pic^1C$ and $\Phi_\CP:\D^b(C) \to \D^b(\Pic^1 C)$ the associated Fourier--Mukai transform.
Then $\CL_E = \Phi_\CP(\CF_E^*)[1]$.
\end{proposition}
\begin{proof}
Indeed, we have
$$
\Ext^\bullet(E,J_L) \cong
\Ext^\bullet(\TE,J_L) \cong 
\Ext^\bullet(\Phi(\CF_E)[1],\Phi(\CL)[1]) = 
\Ext^\bullet(\CF_E,\CL) = 
H^\bullet(C,\CF_E^*\otimes\CL)
$$
and we deduce the first part from Proposition~\ref{tejl}. Moreover, the relative version of the above equality
gives the second part as soon as we observe that the restriction of $\CP$ to the fiber of $C\times\Pic^1C$ over
the point of $\Pic^1C$ corresponding to $\CL$ is $\CL$ itself, so the RHS of the above formula computes
the (derived) restriction of $\Phi_\CP(\CF_E^*)$ to the corresponding point of $\Pic^1C$.
\end{proof}

The above Proposition allows to reinterpret Conjectures~\ref{jl-conj} and~\ref{jlines}.

\begin{corollary}
Assume that for any vector bundle $\CF$ on $C$ of rank $n$ and degree $0$ 
which satisfy~\eqref{cfsd}, \eqref{cfso}, and~\eqref{cfr} one has
% \begin{equation}\label{cfcl}
$\Hom(\CF,\CL) = 0$
% \end{equation} 
for generic $\CL \in \Pic^1(C)$. Then Conjecture~$\ref{jl-conj}$ is true for the Fano threefold $Y_4$.
\end{corollary}
% \begin{proof}
% Since any bundle $\CS_y$ is an extension of an appropriate line bundle $\CL$ of degree $1$ by 
% the line bundle $\CL^{-1}\otimes\xi$ of degree $0$, we see that the first property implies~\eqref{cfso}.
% \end{proof}

% Note that one can construct semistable vector bundles on $C$ satisfying the first two properties of the Corollary.
% For example one can take $\CF_0 = \CR\otimes\CL_{-1}$, where $\CL_{-1}$ is a line bundle of degree $-1$
% (then the first property would hold for $\CF_0$) and then take $\CF = \CF_0 \oplus \tau^*\CF_0^*$
% (then the second property would hold as well). But I don't know whether one can find an example 
% of a vector bundle which satisfies also the third property. If one can this would give an example
% of an instanton with all jumping lines.

On the other hand, one can check that Conjecture~\ref{jlines} is true in this case.

\begin{proposition}
An instanton on $Y_4$ can be reconstructed from the pair $(D_E,\CL_E)$. 
In particular, Conjecture~$\ref{jlines}$ is true for Fano threefolds of degree $4$.
\end{proposition}
\begin{proof}
Since we know that an instanton $E$ can be reconstructed from the associated vector bundle $\CF_E$ on $C$
(Theorem~\ref{miy4}), and since $\CL_E$ is the shift of the Fourier--Mukai image of $\CF_E^*$
with respect to the Fourier--Mukai transform with kernel given by the Poincare bundle, it suffices to check
that one can reconstruct a vector bundle on a curve $C$ from its Fourier--Mukai transform in $\D^b(\Pic^1 C)$.

For this we compute the composition of Fourier--Mukai transforms $\Phi_{\CP^*}\circ\Phi_\CP:\D^b(C) \to \D^b(C)$.
Note that $\Pic^1C$ is a self-dual abelian variety and the Poincare bundle on $C\times \Pic^1C$ is the restriction
of the Poincare bundle from $\Pic^1C\times\Pic^1C$ which is considered as a product of an abelian variety and its dual.
Moreover, since the canonical class of an abelian variety is trivial, the Fourier--Mukai transform 
$\D^b(\Pic^1 C) \to \D^b(\Pic^1 C)$ with the kernel given by the dual of the Poincare bundle is the adjoint (shifted by 2)
of the original Fourier--Mukai functor. Since the Fourier--Mukai functor between the derived categories of $\Pic^1C$ is
an equivalence (see~\cite{Mu}), the composition with the left adjoint functor is the identity, hence the kernel giving the functor 
$\Phi_{\CP^*}\circ\Phi_\CP:\D^b(C) \to \D^b(C)$ is the (derived) restriction of the structure sheaf of the diagonal
on $\Pic^1C\times\Pic^1C$ shifted by $-2$. The above restriction is very easy to compute, it is isomorphic to a cone
of a morphism $\Delta_*\CO_C[-2] \to \Delta_*N^*_{C/\Pic^1C}$ on $C\times C$ (here $\Delta:C \to  C\times C$ is the diagonal
embedding). In particular, it follows that for any vector bundle $F$ on $C$ we have a distinguished triangle
$$
F[-2] \to F\otimes N^*_{C/\Pic^1 C} \to \Phi_{\CP^*}(\Phi_\CP(F)).
$$
Note that the map $F[-2] \to F\otimes N^*$ is given by an element in $\Ext^2(F,F\otimes N^*) = H^2(C,F^*\otimes F\otimes N^*)$.
Since $C$ is a curve this space is zero, hence we have
$$
\Phi_{\CP^*}(\Phi_\CP(F)) \cong F[-1] \oplus F\otimes N^*.
$$
This shows that $F \cong \CH^1(\Phi_{\CP^*}(\Phi_\CP(F)))$ can be reconstructed from $\Phi_\CP(F)$.
Applying this to $F = \CF_E$ we deduce the Proposition.
\end{proof}

% 
% 
%  to the following standard result
% ({\tt OR CONJECTURE??? --- MIHNEA SAYS IT IS TRUE FOR RANK 2 AND FALSE FOR HIGH RANKS!!!}) for the curve $C$.
% 
% \begin{proposition}[{\tt POPA???}]
% If $F$ is a semistable vector bundle of degree $0$ on a curve $C$ of genus $g$ then
% for generic line bundle $L$ of degree $g-1$ one has $H^\bullet(C,F\otimes L) = 0$.
% \end{proposition}
% 
% One deduces immediately the following
% 
% \begin{corollary}
% Conjecture~{\rm\ref{jl-conj}} is true for Fano threefolds of degree $4$.
% \end{corollary}
% 
% Also note that Conjecture~\ref{jlines} can be interpreted as the fact that a semistable vector bundle
% of degree~$0$ on a curve can be reconstructed from its Fourier--Mukai image in $\Pic^{g-1}C$.
% 

\section{Further remarks}\label{s_further}

One can continue research in several directions. First of all one can consider Fano threefolds of index~2 and degree $\le 3$.

\subsection{Fano threefolds of degree 3}

Let $Y_3$ be a Fano threefold of index 2 and degree 3, that is a cubic threefold in $\PP^4$.
There are at least two approaches to the description of the category $\CB_{Y_3}$.
First of all, it is proved in~\cite{K03} that $\CB_{Y_3}$ is equivalent to the nontrivial component
of the derived category of $X_{14}$, a certain Fano threefold of index $1$ and degree $14$ 
which can be associated with $Y_3$ (by the way to construct $X_{14}$ from $Y_3$ one need to choose
a minimal instanton on $Y_3$). So, one can describe instantons on $Y_3$ in terms of vector bundles on $X_{14}$.
This approach may give some interesting results, but it does not look as a way to simplify the question.
The manifold $X_{14}$ does not look more simple than $Y_3$ itself, so it is doubtful that it would be easier
to study vector bundles on $X_{14}$ than on $Y_3$.

Another description of $\CB_{Y_3}$ can be given as follows. Consider a line on $Y_3$ and a projection
from this line $Y_3 \dashrightarrow \PP^2$. It is a conic bundle, so one can consider the associated sheaf $\CC_0$
of even parts of Clifford algebras on $\PP^2$. One can check that $\CB_{Y_3}$ is equivalent to a semiorthogonal
component of the derived category of sheaves of $\CC_0$-modules on $\PP^2$. This is more promising, since
$\PP^2$ has dimension smaller than $Y_3$, so one can hope to have a grip on the structure of the moduli space
of instantons. I would also like to mention that this approach to the description of the category $\CB_{Y_3}$
was used in~\cite{BMMS}.

\subsection{Fano threefolds of degree 2}

Let $Y_2$ be a Fano threefold of index 2 and degree 2, that is a double covering of $\PP^3$
ramified in a smooth quartic surface. Then the category $\CB_{Y_2}$ has the following interesting
property --- its Serre functor is isomorphic to the composition of the shift by 2 with the action
of the involution of the double covering. This behavior is very similar to the behavior of the Serre
functor of Enriques surfaces. And in fact, conjecturally the derived categories of some Enriques
surfaces can be obtained as specializations of $\CB_{Y_2}$ for very special double coverings 
known as Artin--Mumford double solids, see~\cite{IK} for more details. I think it may be interesting
to investigate what kind of moduli space on  Enriques surface appears in this way.

\subsection{Matrix factorizations}

For Fano threefolds which can be described as hypersurfaces in weighted projective spaces 
(i.e. those of degree $3$, $2$ and $1$) the category $\CB_{Y}$ can be also described
as the category of graded matrix factorizations of the equation of the hypersurface, see~\cite{Or09}.
It may be interesting to describe the corresponding moduli spaces of matrix factorizations.

\subsection{Minimal instantons}

Another interesting question is to investigate the moduli spaces of minimal instantons
on Fano threefolds of index $2$. In case of a cubic threefold $Y_3$ this moduli spaces
was investigated in~\cite{MT} and~\cite{K03}. Moreover, it was shown in~\cite{K03} that
in this case minimal instantons provide a relation of cubic threefolds with Fano threefolds
of index $1$ and degree $14$. Because of this it would be very interesting to understand 
the geometry of minimal instantons and their moduli spaces for other $Y_d$.

\end{document}